\documentclass[leqno]{article}

\usepackage[utf8]{inputenc}
\usepackage{amssymb,amstext,amsmath,amscd,amsthm,amsfonts,enumerate,latexsym}
\usepackage{mathtools}
\usepackage{color}
\usepackage{stmaryrd }

\usepackage{relsize}
\usepackage[bbgreekl]{mathbbol}
\usepackage{amsfonts}
\DeclareSymbolFontAlphabet{\mathbb}{AMSb} 
\DeclareSymbolFontAlphabet{\mathbbl}{bbold}
\newcommand{\Prism}{{\mathlarger{\mathbbl{\Delta}}}}

\usepackage{titlesec}
\titleformat{\section}{\centering\large\sc}{\thesection.}{0.7em}{}
\usepackage{etoolbox}
\patchcmd{\abstract}
  {\bfseries\abstractname}
 {\scshape\abstractname}
  {}{}

\usepackage{hyperref}
\hypersetup{
    colorlinks=true,
    linkcolor=blue,
    citecolor=blue,
    filecolor=blue,      
    urlcolor=blue,
}

\usepackage{lipsum}

\usepackage[all]{xy}
\usepackage{tikz-cd}
\begin{document}
\theoremstyle{plain}
\newtheorem{thm}{Theorem}[section]
\newtheorem{theorem}[thm]{Theorem}

\numberwithin{equation}{thm}
\newtheorem*{thm*}{Theorem}
\newtheorem*{cor*}{Corollary}
\newtheorem*{thma}{Theorem A}
\newtheorem*{thmb}{Theorem B}
\newtheorem*{mthm}{Main Theorem}
\newtheorem*{mcor}{Theorem \ref{maincor}}
\newtheorem{thms}[thm]{Theorems}
\newtheorem{prop}[thm]{Proposition}
\newtheorem{proposition}[thm]{Proposition}
\newtheorem{prodef}[thm]{Proposition and Definition}
\newtheorem{lemma}[thm]{Lemma}
\newtheorem{lem}[thm]{Lemma}
\newtheorem{cor}[thm]{Corollary}

\newtheorem{corollary}[thm]{Corollary}
\newtheorem{claim}[thm]{Claim}
\newtheorem*{claim*}{Claim}

\theoremstyle{definition}
\newtheorem{defn}[thm]{Definition}
\newtheorem{construction}[thm]{Construction}
\newtheorem{hypothesis}[thm]{Hypothesis}
\newtheorem{defnProp}[thm]{Definition and Proposition}
\newtheorem{prob}[thm]{Problem}
\newtheorem{question}[thm]{Question}
\newtheorem*{setup}{Set up}
\newtheorem{warning}[thm]{Warning}
\newtheorem{definition}[thm]{Definition}
\newtheorem{ex}[thm]{Example}
\newtheorem{exercise}[thm]{Exercise}
\newtheorem{Example}[thm]{Example}
\newtheorem{example}[thm]{Example}
\newtheorem*{acknowledgements}{Acknowledgements}
\newtheorem*{data}{Data availability}
\newtheorem*{conflict}{Conflict of interest}
\newtheorem{fact}[thm]{Fact}
\newtheorem*{Fact*}{Fact}
\newtheorem{conj}[thm]{Conjecture}
\newtheorem{ques}[thm]{Question}
\newtheorem*{quesa}{Question A}
\newtheorem*{quesb}{Question B}
\newtheorem{case}{Case}
\newtheorem{setting}[thm]{Setting}
\newtheorem{notation}[thm]{Notation}
\newtheorem{rem}[thm]{Remark}
\newtheorem{note}[thm]{Note}
\newtheorem{remark}[thm]{Remark}

\theoremstyle{remark}
\newtheorem*{lemm}{Lemma}
\newtheorem*{pf}{{\sl Proof}}
\newtheorem*{pfpr}{Proof of Proposition \eqref{positive Hecke}}
\newtheorem*{tpf}{{\sl Proof of Theorem 1.1}}
\newtheorem*{cpf1}{{\sl Proof of Claim 1}}
\newtheorem*{cpf2}{{\sl Proof of Claim 2}}

\def\soc{\operatorname{Soc}}
\def\xx{\text{{\boldmath$x$}}}
\def\L{\mathrm{U}}
\def\Z{\mathcal{Z}}
\def\D{\mathcal{D}}
\def\X{\mathcal{X}}
\def\XX{\mathbb{X}}
\def\KK{\mathbb{K}}
\def\bbZ{\mathbb{Z}}
\def\LL{\mathbb{U}}
\def\E{\operatorname{E}}
\def\H{\operatorname{H}}

\def\G{{\sf G}}
\def\T{{\sf T}}

\def\a{\mathfrak a}
\def\b{\mathfrak b}
\def\c{\mathfrak c}
\def\e{\mathrm{e}}
\def\f{\mathfrak{f}}
\def\m{\mathfrak m}
\def\n{\mathfrak n}
\def\p{\mathfrak p}
\def\q{\mathfrak q}
\def\P{\mathfrak P}
\def\Q{\mathfrak Q}
\def\C{\mathcal{C}}
\def\K{\mathrm{K}}
\def\H{\mathrm{H}}
\def\J{\mathrm{J}}
\def\FC{\mathrm F}
\def\Var{\mathrm V}
\def\V{\mathrm V}
\def\r{\mathrm{r}}

\newcommand{\M}{\mathcal{M}}
\newcommand{\N}{\mathcal{N}}

\newcommand{\fkS}{\mathfrak{S}}

\newcommand{\fkZ}{\mathfrak{Z}}

\def\Sp{\operatorname{Spec}}
\def\Spf{\operatorname{Spf}}
\def\Spa{\operatorname{Spa}}

\def\id{\operatorname{id}}
\def\gr{{\rm gr}}

\def\A{{\mathcal A}}
\def\B{{\mathcal B}}
\def\F{{\mathcal F}}
\def\Y{{\mathcal Y}}
\def\W{{\mathcal W}}
\def\M{{\mathcal M}}
\def\N{{\mathcal N}}
\def\O{{\mathcal O}}
\def\H{{\mathcal H}}
\def\G{{\mathcal G}}
\def\R{{\mathcalR}}
\def\I{{\mathcal I}}
\def\L{{\mathcal L}}
\def\U{{\mathcal U}}
\def\V{{\mathcal V}}
\def\fM{\mathfrak M}
\def\fN{\mathfrak N}
\def\fF{\mathfrak F}
\def\T{\mathcal T}
\def\fl{\pi^\flat}
\def\E{{\mathcal E}}
\renewcommand{\R}{\mathcal{R}}

\def\PP{\mathbb{P}}
\def\II{\mathbb{I}}
\newcommand{\vin}{\rotatebox[origin=c]{90}{$\in$}}
\newcommand{\mcF}{\mathcal{F}}
\newcommand{\dg}{\textsuperscript{\textdagger}}

\newcommand{\fil}[1]{\mathrm{Fil}^{#1}}
\newcommand{\bk}{\mathrm{Breuil}\textendash\mathrm{Kisin}}
\newcommand{\bkf}{\mathrm{Breuil}\textendash\mathrm{Kisin}\textendash\mathrm{Fargues}}
\newcommand{\crys}{\mathrm{crys}}
\newcommand{\Ddr}{D_{\mathrm{dR}}(\fM_{\Ao}^{\inf})}
\newcommand{\und}[1]{\underline{#1}}
\newcommand{\BK}{\mathrm{BK}}
\newcommand{\fppf}{\textit{fppf}}
\newcommand{\dcris}{\mathrm{D}_{\mathrm{cris}}}
\newcommand{\ev}{\mathrm{ev}_p}
\title{\Large\bf Prismatic $F$-gauges and a result of T. Liu}
\author{\large Dat Pham}
\date{}
\maketitle
\begin{abstract}
    We give a new proof of a recent result of Tong Liu, which gives a general control on the torsion in the graded pieces of the so-called integral Hodge filtration associated to a crystalline Galois lattice. Our approach is stack-theoretic, and is inspired on the one hand by a result of Gee--Kisin on the shape of mod $p$ crystalline Breuil--Kisin modules, and on the other hand by the structures seen on the diffracted Hodge complex studied by Bhatt--Lurie. Along the way, we also obtain an explicit description of the Hodge--Tate locus in the Nygaard stack $\O_K^{\N}$ for a general extension $K/\mathbf{Q}_p$. 
\end{abstract}
\tableofcontents
\section{Introduction}
\newcommand\myeq{\stackrel{\mathclap{\normalfont\mbox{def}}}{=}}
Let $k/\mathbf{F}_p$ be a perfect field and let $K:=W(k)[1/p]$. Let $T\in\mathrm{Rep}_{\mathbf{Z}_p}(G_K)$ be a crystalline $G_K$-lattice. Fix a choice of uniformizer $\pi\in \O_K$, with Eisenstein polynomial $E(x)$\footnote{The variable is typically denoted by $u$ but here we reserve the notation $u$ for other use.}. Let $\fM$ be the Breuil--Kisin module associated to $T$ and this choice of $\pi$. Pulling the (full $\mathbf{Z}$-indexed) $E(x)$-adic filtration along the Frobenius $\varphi^*\fM\to \fM[1/E(x)]$, and then pushing along the natural map $\varphi^*\fM\twoheadrightarrow \varphi^*\fM/E(x)\varphi^*\fM$ gives an increasing filtration $\fil{\bullet}_{{H}}\fM_{\mathrm{dR}}$ on $\fM_{\mathrm{dR}}:=\varphi^*\fM/E(x)\varphi^*\fM$, which we will refer to as the (integral) Hodge filtration. (The terminology is justified by the fact that, after inverting $p$ this recovers the usual Hodge filtration on $D_{\mathrm{dR}}(T[1/p])\simeq \fM_{\mathrm{dR}}[1/p]$.) It was observed in \cite{GLSunitary} that the torsion in $\mathrm{gr}_{{H}}^{\bullet}\fM_{\mathrm{dR}}$ bears some relation to the shape of the Frobenius acting on $\fM$. For instance, if there is no torsion, then for any choice of $\fkS$-basis $e=(e_1,\ldots,e_n)$ of $\fM$ we have $\varphi(e)=eA\Lambda B$ where $A,B\in \mathrm{GL}_n(\fkS)$ and $\Lambda$ is a diagonal matrix $\mathrm{diag}(E(u)^{r_i})$ (in general one only has $A,B\in \mathrm{GL}_n(\fkS[1/p]$)). An immediate consequence is that in this case the mod $p$ Breuil--Kisin module $\fM/p\fM$ remembers the Hodge--Tate weights of $T[1/p]$. 

The goal of this note is to give a new proof of the following recent result of Tong Liu, which gives a general control on the torsion appearing in $\mathrm{gr}_{{H}}^{\bullet}\fM_{\mathrm{dR}}$. 
\begin{thm}[{\cite[Thm. 1.1]{liuNygaard}}]\label{main thm}
    We have 
    \begin{align*}
        (\mathrm{gr}_{{H}}^i \fM_{\mathrm{dR}})_{\mathrm{tor}}\ne 0\quad\Longrightarrow \quad \text{$i=r+mp$\;\; for some $r\in \mathrm{HT}$ and $m\in\mathbf{Z}_{>0}$},
    \end{align*}
    where $\mathrm{HT}$ denotes the set of Hodge--Tate weights of $T[1/p]$.
\end{thm}
The key new structure to our proof of Theorem \ref{main thm} is a so-called Sen operator $\Theta$ on $\fM_{HT}:=\fM/E(x)\fM$. Our use of $\Theta$ is guided by the structures seen on the diffracted Hodge complex introduced by Bhatt--Lurie in \cite{BhattLurieabsolute}. In turn, we construct $\Theta$ from a certain differential operator $D$. While it is possible to extract $D$ from the \textit{rational} monodromy operator in the classical theory of Breuil--Kisin modules by some delicate approximations (see Section \ref{section GaoLiu}), its existence --- together with the additional symmetries that it satisfies --- seems best explained by the theory of prismatic $F$-gauges by Bhatt--Lurie \cite{Bhatt}. See Subsection \ref{stacky} below. Our considerations here are inspired by a recent result of Gee--Kisin on the shape of mod $p$ crystalline Breuil--Kisin modules; in particular, we follow their strategy and realize the objects of interest as quasi-coherent sheaves on certain stacks.

\begin{remark}
    In fact, we will construct the operator $D$ for a general (possibly ramified) extension $K/\mathbf{Q}_p$. As a key ingredient for this, we extend the explicit description of the Hodge--Tate locus given in \cite[Prop. 5.3.7]{Bhatt} to the case of a general extension; see Proposition \ref{Hodge-Tate description} below.
\end{remark}
\begin{remark}
Our proof is partially inspired by Drinfeld and Bhatt--Lurie's approach to the Deligne--Illusie theorem via the Sen operator (cf. \cite[Remark 4.7.18]{BhattLurieabsolute}). As an illustration of this analogy, note that Theorem \ref{main thm} implies in particular that in the case where $\fM$ is effective (which is the case for representations coming from geometry), $\mathrm{gr}^i M$ is torsion free for all $i<p$. We can thus roughly think of this special case as an incarnation of (a weaker form of) the Deligne--Illusie theorem (with the bound $i<p$ corresponding to the bound $\dim(\mathfrak{X})<p$ in Deligne--Illusie result).    
\end{remark}

\begin{acknowledgements}
The debt that this note owes to the work of Drinfeld \cite{drinfeldprismatization}, Bhatt--Lurie \cite{BhattLurieabsolute, BLprismatization, Bhatt}, and Gee--Kisin will be obvious to the reader. I would also like to thank Toby Gee and Bao Le Hung for helpful discussions, as well as Toby Gee, Arthur-C\' esar Le Bras, and Tong Liu for their comments. After writing an initial draft of this note, I learned that Gao--Liu and Gee--Kisin have also independently found related proofs of Theorem \ref{main thm}. I am very grateful to them for informing me of their work and for kindly coordinating in announcing our results. This work was supported by the Simons Collaboration on Perfection in Algebra, Geometry, and Topology. 

\end{acknowledgements}
\section{Proof of Main Theorem}
To avoid distractions, we will first prove a more general result (Proposition \ref{main prop}) by isolating the key input. In Subsection \ref{stacky} below we will indicate how it specializes to the situation of Theorem \ref{main thm}. \\

\noindent\textbf{Set up}. Consider an {increasing} (honest) filtration of finite free $\O_K$-modules 
\begin{align*}
    \mathrm{Fil}_{\bullet}^{\mathrm{conj}}: \quad \ldots \subseteq \mathrm{Fil}_0^{\mathrm{conj}}\subseteq \ldots \subseteq \mathrm{Fil}_i^{\mathrm{conj}}\subseteq \ldots.
\end{align*}
We assume that this is a finite filtration, i.e., $\mathrm{Fil}_i^{\mathrm{conj}}$ stabilizes for $i\gg 0$, and is $0$ for $i\ll 0$. \\

\noindent\textbf{Hypothesis.} Assume there is a filtered endomorphism $\Theta: \mathrm{Fil}_{\bullet}^{\mathrm{conj}}\to \mathrm{Fil}_{\bullet}^{\mathrm{conj}}$ with the property that $\Theta$ acts on the $i$th graded piece $\mathrm{gr}_i^{\mathrm{conj}}$ via multiplication by $-i$. (The superscript ``conj'' stands for ``conjugate''.)\\

Inverting $p$ gives a filtration of $K$-vector spaces, and we define as usual its Hodge--Tate weights as the set of filtration jumps, i.e.,
\begin{displaymath}
    \mathrm{HT}:=\{i\in\mathbf{Z}\;|\;\mathrm{gr}_i^{\mathrm{conj}}[1/p]\ne 0\}.
\end{displaymath}
\begin{prop}\label{main prop}
    We have 
    \begin{displaymath}
        (\mathrm{gr}_i^{\mathrm{conj}})_{\mathrm{tor}}\ne 0\quad\Longrightarrow \quad \text{$i=r+mp$\;\; for some $r\in \mathrm{HT}$ and $m>0$}.
    \end{displaymath}
\end{prop}
\begin{proof}
    Set 
\begin{align*}
 I:=\{r+mp\;|\; r\in \mathrm{HT}, m>0\}.   
\end{align*}
We need to show that if $i\notin I$, then $\mathrm{gr}_i^{\mathrm{conj}}$ is $\O_K$-free. We will do this by induction on $i$. If $i\ll 0$, then $\mathrm{gr}_i^{\mathrm{conj}}=0$ and there is nothing to prove. Assume the result for $i'<i$ (and $i'\notin I$), we now deduce it for $i$. In fact we will show the stronger assertion that the sequence 
\begin{align*}
    0\to \mathrm{Fil}_{i-1}^{\mathrm{conj}}\to \mathrm{Fil}_i^{\mathrm{conj}}\to \mathrm{gr}_i^{\mathrm{conj}}\to 0
\end{align*}
of $\O_K[\Theta]$-modules splits (as $\mathrm{Fil}_i^{\mathrm{conj}}$ is $\O_K$-free, this implies in particular that $\mathrm{gr}_i^{\mathrm{conj}}$ is free, as wanted), which in turn will follow from
\begin{displaymath}
    \mathrm{Ext}^1_{\O_K[\Theta]}(\mathrm{gr}_i^{\mathrm{conj}},\mathrm{Fil}_{i-1}^{\mathrm{conj}})=0.
\end{displaymath}
By d\' evissage, it suffices to show 
\begin{displaymath}
    \mathrm{Ext}^1_{\O_K[\Theta]}(\mathrm{gr}_i^{\mathrm{conj}},\mathrm{gr}_{j}^{\mathrm{conj}})=0\quad\text{for each $j<i$}.
\end{displaymath}
We consider two cases. If $p\nmid i-j$, then we are done as the LHS is killed by $(\Theta+i)-(\Theta+j)=i-j$, a unit. If $p|i-j$, then by definition of the set $I$, $j\notin \mathrm{HT}$ and we still have $j\notin I$. Thus, $\mathrm{gr}_j^{\mathrm{conj}}[1/p]=0$ but also $\mathrm{gr}_j^{\mathrm{conj}}$ is $p$-flat by the inductive hypothesis for $j$. This forces $\mathrm{gr}_j^{\mathrm{conj}}=0$, and the result trivially holds. (Note that the same argument also shows that $\mathrm{Hom}_{\O_K[\Theta]}(\mathrm{gr}_i^{\mathrm{conj}},\mathrm{Fil}_{i-1}^{\mathrm{conj}})=0$, i.e. the splitting is unique.)  
\end{proof} 
\subsection{A stacky perspective}\label{stacky}
After Proposition \ref{main prop}, to finish the proof of Theorem \ref{main thm} we need to construct an increasing filtration $\mathrm{Fil}_\bullet^{\mathrm{conj}}$ together with an endomorphism $\Theta$ as above with the additional property that $\mathrm{gr}_{\bullet}^{\mathrm{conj}}\simeq \mathrm{gr}_{\mathrm{H}}^{\bullet} \fM_{\mathrm{dR}}$. 

Our construction of $\mathrm{Fil}_\bullet^{\mathrm{conj}}$ and $\Theta$ is guided by the structure seen on the diffracted Hodge complex studied by Bhatt--Lurie in \cite[\S 4.7]{BhattLurieabsolute}, and is explained in \cite[Remark 6.5.11]{Bhatt} in a geometric context. The material in this subsection is therefore presumably well-known to the experts, although we do not know of a treatment in the literature in the level of generality that we require.

\subsubsection*{Construction of $\mathrm{Fil}_{\bullet}^{\mathrm{conj}}$ and $\Theta$}
Consider the so-called conjugate filtration
\begin{displaymath}
    \mathrm{Fil}_{\bullet}^{\mathrm{conj}}\fM_{HT}: \quad \ldots\hookrightarrow \underbrace{\fil{i-1}\varphi^*\fM/\fil{i}\varphi^*\fM}_{\mathrm{Fil}_{i-1}^{\mathrm{conj}}\fM_{HT}}\xhookrightarrow{u} \underbrace{\fil{i}\varphi^*\fM/\fil{i+1}\varphi^*\fM}_{\mathrm{Fil}_i^{\mathrm{conj}}\fM_{HT}}\xhookrightarrow{} \ldots,
\end{displaymath}
where the transition map $u$ is induced by the multiplication by $E(x)$. One checks easily that this is a finite increasing filtration of finite free $\O_K$-modules, whose underlying non-filtered module is $\fM_{HT}:=\fM/E(x)\fM$ (justifying the notation). Moreover there is a natural graded isomorphism 
\begin{displaymath}
 \mathrm{gr}_{\bullet}^{\mathrm{conj}}\fM_{HT}\simeq \mathrm{gr}^{\bullet} \fM_{dR}.   
\end{displaymath}
(This identification also admits a geometric explanation; see Corollary \ref{identfy Hodge filtration} below.) In what follows, we will often omit $\fM_{HT}$ from the notation and simply write $\mathrm{Fil}^{\mathrm{conj}}_{\bullet}$, etc. We now explain the construction of $\Theta$. 

\textbf{From now on we {drop} the assumption that $K/\mathbf{Q}_p$ is unramified, i.e. we allow $K/\mathbf{Q}_p$ to be any complete discretely valued extension with perfect residue field.}

We will construct an operator $D: \mathrm{Fil}_{\bullet}\to \mathrm{Fil}_{\bullet}[-1]$ with the property that
\begin{align}\label{commutation Du uD}
Du-uD=E'(\pi).    
\end{align}
Once $D$ is constructed, we define $\Theta$ by setting
\begin{equation}\label{intro relation D theta}
\Theta:=uD-iE'(\pi)\quad\text{on $\mathrm{Fil}_i^{\mathrm{conj}}$}.
\end{equation}
(We refer the reader to Lemma \ref{identify Sen} below for a justification of this formula.) It follows from \eqref{commutation Du uD} that $\Theta$ is indeed a filtered morphism. As $uD$ on $\mathrm{Fil}_i^{\mathrm{conj}}$ factors through $u: \mathrm{Fil}_{i-1}^{\mathrm{conj}}\hookrightarrow \mathrm{Fil}_i^{\mathrm{conj}}$ (so induces the zero map on $\mathrm{gr}_i^{\mathrm{conj}}$), we also have that $\Theta$ acts on $\mathrm{gr}_i^{\mathrm{conj}}$ via multiplication by $-iE'(\pi)$. If $K/\mathbf{Q}_p$ is unramified, then clearly $E'(\pi)=1$; hence in this case $\Theta$ acts as $-i$ on $\mathrm{gr}_i^{\mathrm{conj}}$, as desired.

Thus, it remains to construct $D$. As mentioned in the Introduction, while one can extract $D$ from the \textit{rational} monodromy operator in the classical theory of Breuil--Kisin modules, its existence seems best explained by the theory of prismatic $F$-gauges by Bhatt--Lurie, as we now explain. 

More precisely, the existence of $D$ and its relation to $\Theta$ (as given by \eqref{intro relation D theta}) are explained by the following commutative diagrams of stacks over $\Spf(\mathbf{Z}_p)$:
\begin{equation}\label{diagram stack}
    \begin{tikzcd}
    (B\mathbf{G}_m)_{\O_K}\ar[d,hook,"u=0"']\ar[r,hook,"t=0"] &(\mathbf{A}_{+}^1/\mathbf{G}_m)_{\O_K}\ar[r,"i_{dR}"] & \O_K^{\N}\\
    (\mathbf{A}_{-}^1/\mathbf{G}_m)_{\O_K}\ar[r,"can"] & (\mathbf{A}_{-}^1/\mathbf{G}_a^{\mathrm{\sharp}}\rtimes\mathbf{G}_m)_{\O_K}\ar[r,"\pi_{\O_K}","\simeq"'] & (\O_K^{\N})_{t=0}\ar[u,hook,"t=0"]\\
    \Spf(\O_K)=(\mathbf{G}_m/\mathbf{G}_m)_{\O_K}\arrow[rr, bend right=13, "\overline{\rho}_{(\fkS,I)}"']\ar[u,hook,"u\ne 0"]\ar[r,"can"] & (\mathbf{G}_m/\mathbf{G}_a^{\mathrm{\sharp}}\rtimes\mathbf{G}_m)_{\O_K}\ar[u,hook,"u\ne 0"]\ar[r,"\simeq"] &\O_K^{HT}\ar[u,hook,"j_{HT}"]. 
    \end{tikzcd}
\end{equation}
Let us briefly explain the objects appearing in the diagram.
\begin{itemize}
    \item $\mathbf{A}_{+}^1$ (resp. $\mathbf{A}_{-}$) denotes the affine line $\mathbf{A}^1$ where the coordinate $t$ (resp. $u$) is placed in grading degree $1$ (resp. $-1$). Furthermore, we let $\mathbf{G}_a^{\mathrm{\sharp}}$ act on $\mathbf{A}_{-}^1$ by $a\cdot_{\O_K}x:=E'(\pi)a+x$; one checks that this then extends to an action of $\mathbf{G}_a^{\mathrm{\sharp}}\rtimes \mathbf{G}_m$, where the semidirect product is formed by letting $\mathbf{G}_m$ act on $\mathbf{G}_a^{\mathrm{\sharp}}$ by $(\lambda,a)\mapsto \lambda^{-1}a$.
    \item $\O_K^{\N}$ is the filtered prismatization of $\Spf(\O_K)$. This is a filtered stack, i.e., comes with a map $t: \O_K^{\N}\to (\mathbf{A}_{+}^1/\mathbf{G}_m)_{\mathbf{Z}_p}$. Moreover, there is an open embedding $j_{HT}: \O_K^{\Prism}\hookrightarrow \O_K^{\N}$ from the prismatization $\O_K^{\Prism}$. See \cite[\textsection 5.3]{Bhatt} for more details. The map $i_{dR}$ in the top row is the de Rham map defined in \cite[Construction 5.3.13]{Bhatt}. The map $\overline{\rho}_{(\fkS,I)}$ is the usual map associated to the chosen Breuil--Kisin prism $(\fkS,I)$, viewed as an object in the absolute prismatic site of $\Spf(\O_K)$. See \cite[Construction 3.10]{BLprismatization}. 
\item The isomorphism $ (\mathbf{A}_{-}^1/\mathbf{G}_a^{\mathrm{\sharp}}\rtimes \mathbf{G}_m)_{\O_K}\simeq (\O_K^{\N})_{t=0}$ in the middle row will be defined and proved in Subsection \ref{section HT description} below.
\end{itemize}
Finally, we explain the relevance of diagram \eqref{diagram stack} to the construction of the operators $D$ and $\Theta$. More details will be given in Section \ref{section identification} below.
\begin{itemize}
    \item The graded isomorphism $\mathrm{gr}_{\bullet}^{\mathrm{conj}}\fM_{\mathrm{HT}}\simeq \mathrm{gr}_{\mathrm{H}}^{\bullet}\fM_{dR}$ results from commutativity of the top rectangle. See Subsection \ref{subsection de Rham} below.
    \item It follows from the isomorphism $(\O_K^{\N})_{t=0}\simeq (\mathbf{A}_{-}^1/\mathbf{G}_a^{\mathrm{\sharp}}\rtimes \mathbf{G}_m)_{\O_K}$ that quasi-coherent sheaves on $(\O_K^{\N})_{t=0}$ are equivalent to $p$-complete graded modules over the ($p$-completed) Weyl algebra $\O_K\{u,D\}/(Du-uD-E'(\pi))$, i.e. $p$-complete graded $\O_K[u]$-modules $M$ together with a graded endomorphism $D: M\to M[-1]$ satisfying the commutation relation $Du-uD=E'(\pi)$ (and a certain nilpotent condition), cf. Lemma \ref{derivation}. Pulling back along the natural map $[\mathbf{A}_{-}^1/\mathbf{G}_m]\to (\O_K^{\N})_{t=0}$ then simply amounts to forgetting the derivation. Thus if $E\in \mathrm{Coh}(\O_K^{\mathrm{Syn}})$ denotes the $F$-gauge associated to the given crystalline lattice $T$, then $E|_{(\O_K^{\N})_{t=0}}$ gives rise to a derivation $D$ (satisfying the desired properties) on the filtration corresponding (via the Rees construction) to $E|_{[\mathbf{A}_{-}^1/\mathbf{G}_m]}$. We then shows in Lemma \ref{conjugate filtration} below that the latter filtration is nothing but with the conjugate filtration $\mathrm{Fil}_{\bullet}^{\mathrm{conj}}$ introduced above. This finishes the construction of $D$.
    \item It is known that pulling back along the map $\overline{\rho}_{(\fkS,I)}$ lifts to an equivalence between quasi-coherent sheaves on $\O_K^{HT}$ and $p$-complete modules over $\O_K$ equipped with a so-called Sen operator $\Theta$ (see Theorem \ref{BL Sen} below). In Lemma \ref{identify Sen} below, we will show that under this equivalence, the restriction $E|_{\O_K^{HT}}$ corresponds to the module $\varinjlim \mathrm{Fil}^{\mathrm{conj}}_i$ underlying the conjugate filtration equipped with the Sen operator given by $\Theta=uD-iE'(\pi)$ on $\mathrm{Fil}_i^{\mathrm{conj}}$. This explains the above construction of $\Theta$ in terms of $D$, as given in \eqref{intro relation D theta}. 
\end{itemize}\hfill$\square$
\section{Identifications with the stacky picture}\label{section identification}
In this section, we justify the various identifications with the stacky picture, as alluded to in Subsection \ref{stacky}. In particular, we extend the explicit description of the Hodge--Tate locus given in \cite[Prop. 5.3.7]{Bhatt} from the case $K=\mathbf{Q}_p$ to the case of a general extension  $K/\mathbf{Q}_p$; see Proposition \ref{Hodge-Tate description} below.  As mentioned earlier, the materials here will not surprise an expert; however since the proofs are not available in the literature, we work them out here for completeness.
\subsection{Preliminaries}
Let us begin by recalling the relation, as discussed in \cite{Bhatt}, between crystalline Galois lattices and coherent sheaves on $\O_K^{\mathrm{Syn}}$. 

Let $X$ be a quasi-syntomic $p$-adic formal scheme. Recall that for each object $(A,I)\in X_{\Prism}$ in the absolute prismatic site of $X$, there is an associated morphism $\rho_{(A,I)}: \Spf(A)\to X^{\Prism}$ (see \cite[Construction 3.10]{BLprismatization}). By \cite[Prop. 8.15]{BLprismatization}, pulling back along these maps gives an equivalence 
\begin{displaymath}
    \mathrm{Perf}(X^{\Prism})\simeq \varprojlim_{(A,I)\in X_{\Prism}}\mathrm{Perf}(A)=:\mathrm{Perf}(X_{\Prism},\O_{\Prism}) 
\end{displaymath}
onto the category of prismatic crystals in perfect complexes on $X$. 

By definition of $X^{\mathrm{Syn}}$ as a coequalizer, there is a natural (\' etale) map $X^{\Prism}\to X^{\mathrm{Syn}}$. Restricting along this and using the equivalence above, we obtain a functor 
\begin{displaymath}
    \mathrm{Perf}(X^{\mathrm{Syn}})\to \mathrm{Perf}(X^{\Prism})\simeq \mathrm{Perf}(X_{\Prism},\O_{\Prism}),
\end{displaymath}
which in fact naturally lifts to a functor 
\begin{displaymath}
    \mathrm{Perf}(X^{\mathrm{Syn}})\to \mathrm{Perf}^{\varphi}(X_{\Prism},\O_{\Prism})
\end{displaymath}
into the category of prismatic $F$-crystals in perfect complexes on $X$. To see this, note that given any $E\in \mathrm{Perf}(X^{\N})$, there is a natural correspondence
\begin{displaymath}
    j_{HT}^*E \xleftarrow{a} \varphi^*\pi_*E\xrightarrow{b}\varphi^*j_{dR}^*E. 
\end{displaymath}
Namely, $a$ (resp. $b$) comes from adjunction and the identity $\pi\circ j_{HT}=\varphi$ (resp. $\pi\circ j_{dR}=\mathrm{id}$). 
As $M$ is perfect, the maps $a$ and $b$ are in fact $I$-isogenies (where $I\subseteq \O_{X^{\Prism}}$ is the Hodge--Tate ideal sheaf), and so we obtain a natural (in $E$) isomorphism 
\begin{displaymath}
\iota_E: \varphi^*(j_{dR}^*E)[1/I]\simeq j_{HT}^*E[1/I].    
\end{displaymath}
Now lifting $E$ to an object in $\mathrm{Perf}(X^{\mathrm{Syn}})$ amounts to specifying an isomorphism $j_{HT}^*E\simeq j_{dR}^*E$, and so we obtain the desired functor
\begin{displaymath}
    \mathrm{Perf}(X^{\mathrm{Syn}})\to \mathrm{Perf}^{\varphi}(X_{\Prism},\O_{\Prism}). 
\end{displaymath}
See \cite[\textsection 6.3]{Bhatt} for more details. 

Assume now that $X=\Spf(\O_X)$. Let $(A_{\inf},(\xi))$ be the perfect prism associated to $\O_C$ (where $\xi$ is a generator of the kernel of the (non-twisted) Fontaine's theta map). By the preceding discussion, there is a natural functor 
\begin{align}\label{O_C reflexive}
    \mathrm{Perf}(\O_C^{\N}) & \to \{(N,M,\iota)\;\text{where $N,M\in\mathrm{Perf}(A_{\inf})$ and $\iota: N[1/\xi]\simeq M[1/\xi]$}\}\\
    E &\mapsto (\varphi^*(j_{dR}^*E),j_{HT}^*E,\iota_E)\nonumber.
\end{align}
In \cite[\textsection 6.6.1]{Bhatt}, Bhatt isolates a subcategory $\mathrm{Coh}^{\mathrm{refl}}(\O_C^{\N})$ of $\mathrm{Perf}(\O_C^{\N})$ with the property that the functor \eqref{O_C reflexive} restricts to an equivalence 
\begin{align}\label{equiv reflexive O_C N}
    \mathrm{Coh}^{\mathrm{refl}}(\O_C^{\N}) &\simeq \{(N,M,\iota)\;\text{where $N,M\in\mathrm{Vect}(A_{\inf})$ and $\iota: N[1/\xi]\simeq M[1/\xi]$}\}.
\end{align}
Now by \cite[Prop. 5.5.8]{Bhatt} (see also \cite[Exa. 5.5.6]{Bhatt}), there is an isomorphism $\R(\varphi^{-1}(\xi)^{\bullet}A_{\inf})\simeq \O_C^{\N}$, where the LHS denotes the $(p,\xi)$-completed Rees stack for the Nygaard filtration $\varphi^{-1}(\xi)^{\bullet}A_{\inf}$. Composing further with the isomorphism $\varphi: \varphi^{-1}(\xi)^{\bullet}A_{\inf}\simeq \xi^{\bullet}A_{\inf}$, we obtain an isomorphism 
\begin{align}\label{the map pi O_C}
    \pi_{\O_C}: \Spf(A_{\inf}[u,t]/(ut-\xi))/\mathbf{G}_m\simeq \R(\xi^{\bullet}A_{\inf}) &\simeq \O_C^{\N},
\end{align}
where as usual $t$ is the degree 1 Rees parameter and $u$ has degree $-1$. 

For later use we note the following result. 
\begin{lemma}\label{useful lemma identify filtration}
    Let $E\in\mathrm{Coh}^{\mathrm{refl}}(\O_C^{\N})$. Then the filtration over $\xi^{\bullet}A_{\inf}$ corresponding to $\pi_{\O_C}^*E$ (via the Rees dictionary) can be recovered from the tuple $(N,M,\iota)$ associated to $E$ as $N\cap \xi^{\mathbf{Z}}M=\varphi^*(j_{dR}^*E)\cap \xi^{\mathbf{Z}}(j_{HT}^*E)$; here $\xi^{\mathbf{Z}}M$ denotes the full $\mathbf{Z}$-indexed $\xi$-adic filtration on $M[1/\xi]$. (In particular, this is an honest filtration.)
\end{lemma}
\begin{proof}
    This follows from the construction of the quasi-inverse functor in the proof of \cite[Prop. 6.6.3]{Bhatt}. (Strictly speaking, in \cite[Prop. 6.6.3]{Bhatt} Bhatt defines the corresponding subcategory in $\mathrm{Perf}_{\mathrm{gr}}(A_{\inf}[u,t]/(ut-\xi)$ (rather than directly in $\mathrm{Perf}(\O_C^{\N})$), but the reader can check that our presentation above is precisely consistent with Bhatt's.)
\end{proof}
\begin{defn}[{cf. \cite[Defn. 6.6.4]{Bhatt}}]\label{reflexive gauge O_C}
Define the category $\mathrm{Coh}^{\mathrm{refl}}(\O_C^{\mathrm{Syn}})$ of reflexive $F$-gauges on $\O_C$ to be the full subcategory of $\mathrm{Perf}(\O_C^{\mathrm{Syn}})$ consisting of $E$'s such that $E|_{\O_C^{\N}}$ belongs to $\mathrm{Coh}^{\mathrm{ref}}(\O_C^{\N})$.    
\end{defn}
By \eqref{equiv reflexive O_C N}, restricting along $\O_C^{\Prism}\to \O_C^{\mathrm{Syn}}$  yields an equivalence    
\begin{align*}
    \mathrm{Coh}^{\mathrm{refl}}(\O_C^{\mathrm{Syn}})&\simeq \mathrm{Vect}^{\varphi}((\O_C)_{\Prism},\O_{\Prism})
\end{align*}
onto the category of prismatic $F$-crystals in vector bundles on $\O_C$ (cf. \cite[Cor. 6.6.5]{Bhatt}). 

\begin{defn}[{cf. \cite[Defn. 6.6.11]{Bhatt}}]
    An $F$-gauge $E\in \mathrm{Perf}(\O_K^{\mathrm{Syn}})$ is called reflexive if $E|_{\O_C^{\mathrm{Syn}}}$ is reflexive
in the sense of Definition \ref{reflexive gauge O_C}. Write $\mathrm{Coh}^{\mathrm{refl}}(\O_K^{\mathrm{Syn}})$ for the full subcategory spanned by such $F$-gauges.
\end{defn} 
By design, restriction again defines a functor  
\begin{displaymath}
    \mathrm{Coh}^{\mathrm{ref}}(\O_K^{\mathrm{Syn}}) \to \mathrm{Vect}^{\varphi}((\O_K)_{\Prism},\O_{\Prism}) 
\end{displaymath}
which turns out to be an equivalence by \cite[Thm. 6.6.13]{Bhatt}. Combining with the main result of \cite{prismatic} then gives an equivalence 
\begin{displaymath}
    \mathrm{Coh}^{\mathrm{ref}}(\O_K^{\mathrm{Syn}}) \simeq \mathrm{Vect}^{\varphi}((\O_K)_{\Prism},\O_{\Prism})\simeq \mathrm{Rep}_{\mathbf{Z}_p}^{\mathrm{cris}}(G_K).
\end{displaymath}
In particular it makes sense to talk about the $F$-gauge $E\in \mathrm{Coh}^{\mathrm{ref}}(\O_K^{\mathrm{Syn}})$ associated to the given crystalline lattice $T$.

\subsection{Explicit description of the Hodge--Tate locus}\label{section HT description}
The main result in this subsection is Proposition \ref{Hodge-Tate description}, which gives an explicit presentation of the Hodge--Tate locus $(\O_K^{\N})_{t=0}$ as a quotient stack. This extends \cite[Prop. 5.3.7]{Bhatt}, which treats the case $\O_K=\mathbf{Z}_p$. 

We first recall a general construction from \cite[Rem. 5.5.19]{Bhatt}. Namely, given any prism $(A,I)$ and any map $\Spf(A/I)\to X$ of bounded $p$-adic formal schemes, there is a natural map of filtered stacks 
\begin{align*}
 \pi_X: \R(I^{\bullet}A)\to {X}^{\N},   
\end{align*}
where $\R(I^{\bullet}A)$ denotes the $(p,I)$-completed Rees stack of the $I$-adic filtration on $A$. Note that \textit{loc. cit.} seems to assume more than just this data, but this is all that is needed to \textit{construct} the map, as we now recall\footnote{The map $\pi_X$ depends on the prism $(A,I)$ but we omit it for ease of notation. It will also be clear that the map $\pi_{\O_C}$ from \eqref{the map pi O_C} above is a special case of this construction, justifying our notation.}. Recall that, given a $p$-nilpotent test ring $S$, a point $x\in \R(I^{\bullet}A)(S)$ is given by a map $A\to S$ that kills some power of $I$, a line bundle $L\in  \mathrm{Pic}(S)$, and a factorization $I\otimes_A S \xrightarrow{u} L\xrightarrow{t}S$ of the canonical map. As usual, the map $A\to S$ lifts uniquely to give a $\delta$-$A$-algebra structure on the Witt ring scheme $W$ over $S$, and one can consider the commutative diagram with exact rows
\begin{equation}\label{big diagram pi_X general}
    \begin{tikzcd}
                    0\ar[r] & I\otimes_A\mathbf{G}_a^{\mathrm{\sharp}}\arrow[bend right=35,red]{dd}[near start,swap]{can}\ar[r]\ar[d,"u^{\mathrm{\sharp}}"] & I\otimes_A W\arrow[bend right=35,red]{dd}[near start,swap]{can}\ar[d]\ar[r] & I\otimes_A F_*W \ar[d,equal]\ar[r] & 0\\
    0\ar[r] & \mathbf{G}_a^{\mathrm{\sharp}}\ar[d,"t^{\mathrm{\sharp}}"]\ar[r] & M_u\ar[d,red,"d_{u,t}"] \ar[r] & I
    \otimes F_*W \ar[r]\ar[d,"can"] & 0\\
    0\ar[r] & \mathbf{G}_a^{\mathrm{\sharp}}\ar[r] & W\ar[r] & F_*W\ar[r] & 0.
        \end{tikzcd}
\end{equation}
Note that the middle arrow defines a map $A/I\to (W/M_u)(S)$ of animated rings, so the filtered Cartier--Witt divisor $M_u\xrightarrow{d_{u,t}} W$ naturally lifts to a point $\pi_{X}(x)\in X^{\N}(S)$, as wanted.

We now apply this construction to the case $X=\Spf(\O_K)$ and $(A,I)=(W(k)[[x]],E(x))$, our fixed Breuil--Kisin prism. In this case, as explained in \textit{loc. cit.}, the map $\R(I^{\bullet}A)\xrightarrow{\pi_{\O_K}}\O_K^{\N}$ is in fact a flat cover. The choice of the generator $E(x)$ of $I$ identifies $\R(I^{\bullet}A)_{t=0}\simeq (\mathbf{A}_{-}^1/\mathbf{G}_m)_{\O_K}$. Explicitly, a point $(S\xrightarrow{u}L)\in (\mathbf{A}_{-}^1/\mathbf{G}_m)(S)$ corresponds to the point $I\otimes_AS\xrightarrow{u}L\xrightarrow{t=0}S$ of $\R(I^{\bullet}A)_{t=0}$; here we view $u$ as a map $I\otimes_AS\to S$ via the trivialization $I=E(x)A\simeq A$. 
    
    Thus $\pi_{\O_K}$ restricts to a map
    \begin{align}\label{map from A^1 to HT}
     \pi_{\O_K}: (\mathbf{A}_{-}^1)_{\O_K}\to (\mathbf{A}_{-}^1/\mathbf{G}_m)_{\O_K} \to (\O_K^N)_{t=0}   
    \end{align}
    of stacks over $\Spf(\O_K)$. Of course this is still a flat surjection.
\begin{prop}\label{Hodge-Tate description}
   The map \eqref{map from A^1 to HT} factors through an isomorphism
   \begin{align*}
        (\mathbf{A}^1_{-}/\mathbf{G}_a^{\mathrm{\sharp}}\rtimes \mathbf{G}_m)_{\O_K}\simeq (\O_K^{\N})_{t=0}.
    \end{align*}
   of stacks over $\Spf(\O_K)$. Here the action of $\mathbf{G}_a^{\mathrm{\sharp}}\rtimes \mathbf{G}_m$ on $\mathbf{A}_{-}^1$ is given by $(a,\lambda)\cdot_{\O_K} u:=E'(\pi)a+\lambda^{-1}u$\footnote{The appearance of $\lambda^{-1}$ (rather than $\lambda$) is simply due to our convention that the coordinate $u$ of $\mathbf{A}_{-}^1$ has degree $-1$.}.
\end{prop}
We begin with some preparations. The following is simply an elaboration of \cite[Prop. 5.2.1 (2)]{Bhatt}. 
\begin{lemma}[{{\cite[Prop. 5.2.1]{Bhatt}}}]\label{elaboration lemma}
    Applying ${\mathrm{Hom}}_W(-,\mathbf{G}_a^{\mathrm{\sharp}})$ to the standard sequence $0\to I\otimes_A \mathbf{G}_a^{\mathrm{\sharp}}\to I\otimes_A W\to I\otimes_A F_*W\to 0$ gives an exact sequence
\begin{align}\label{pushout lemma}
    {\mathrm{Hom}}_W(I\otimes_A W,\mathbf{G}_a^{\mathrm{\sharp}})\to {\mathrm{Hom}}_W(I\otimes_A\mathbf{G}_a^{\mathrm{\sharp}},\mathbf{G}_a^{\mathrm{\sharp}})\to \mathrm{Ext}^1_W(I\otimes_A F_*W,\mathbf{G}_a^{\mathrm{\sharp}})\to 0.
\end{align}
Using the trivialization $I=E(x)A\simeq A$, we identify 
\begin{align*}
   \mathbf{G}_a^{\mathrm{\sharp}}(S) &\simeq  {\mathrm{Hom}}_W(I\otimes_A W,\mathbf{G}_a^{\mathrm{\sharp}})\\
    a &\mapsto (E(x)\otimes w\mapsto wa),
\end{align*}
and 
\begin{align*}
   S=\mathbf{A}_{-}^1(S) &\simeq  {\mathrm{Hom}}_W(I\otimes_A \mathbf{G}_a^{\mathrm{\sharp}},\mathbf{G}_a^{\mathrm{\sharp}})\\
    u &\mapsto (E(x)\otimes a\xrightarrow{u^{\mathrm{\sharp}}} ua);
\end{align*}
Under these identifications, the first map in \eqref{pushout lemma} identifies with the natural map $\mathbf{G}_a^{\mathrm{\sharp}}(S)\to \mathbf{A}_{-}^1(S)=S$, and the second map takes a point $u\in S$ to the pushout 
\begin{displaymath}
    \begin{tikzcd}
         0\ar[r] & I\otimes_A\mathbf{G}_a^{\mathrm{\sharp}}\ar[r]\ar[d,"u^{\mathrm{\sharp}}"] & I\otimes_A W\ar[d]\ar[r] & I\otimes_A F_*W\ar[d,equal]\ar[r] & 0\\
    0\ar[r] & \mathbf{G}_a^{\mathrm{\sharp}}\ar[r] & M_u\simeq \tfrac{\mathbf{G}_a^{\mathrm{\sharp}}\oplus (I\otimes_A W)}{\{(u^{\mathrm{\sharp}}(x),-x)|x\in I\otimes \mathbf{G}_a^{\mathrm{\sharp}}\}} \ar[r] & I
    \otimes F_*W \ar[r] & 0.
    \end{tikzcd}
\end{displaymath}
Furthermore, given $a\in\mathbf{G}_a^{\mathrm{\sharp}}(S)$, the associated isomorphism $\iota(a): M_{u}\simeq M_{a+u}$ of extensions is given by 
\begin{align}\label{iso extensions elaboration}
    \iota(a): M_u\simeq \tfrac{\mathbf{G}_a^{\mathrm{\sharp}}\oplus (I\otimes_A W)}{\{(u^{\mathrm{\sharp}}(x),-x)|x\in I\otimes \mathbf{G}_a^{\mathrm{\sharp}}\}} &\to  \tfrac{\mathbf{G}_a^{\mathrm{\sharp}}\oplus (I\otimes_A W)}{\{(a+u)^{\mathrm{\sharp}}(x),-x)|x\in I\otimes\mathbf{G}_a^{\mathrm{\sharp}}\}}\simeq M_{a+u}\nonumber \\
     (x,y) & \mapsto (x-a(y),y).
\end{align}
\end{lemma}
\begin{lemma}[{{A twisted version of \cite[Prop. 5.3.7]{Bhatt}}}]\label{twisted version lemma}
    The composition
    \begin{displaymath}
    (\mathbf{A}_{-}^1)_{\O_K}\xrightarrow{\pi_{\O_K}} (\O_K^{\N})_{t=0}\to (\mathbf{Z}_p^{\N})_{t=0}\times \Spf(\O_K)    
    \end{displaymath}
    factors through an isomorphism 
    \begin{align}
        (\mathbf{A}^1_{-}/\mathbf{G}_a^{\mathrm{\sharp}}\rtimes \mathbf{G}_m)_{\O_K}\simeq (\mathbf{Z}_p^{\N})_{t=0}\times\Spf(\O_K).
    \end{align}
    Here the action of $\mathbf{G}_a^{\mathrm{\sharp}}\rtimes \mathbf{G}_m$ on $\mathbf{A}_{-}^1$ is given by $(a,\lambda)\cdot_{\mathbf{Z}_p} u:=a+\lambda^{-1}u$.
\end{lemma}
\begin{proof}
    The proof is similar to that of \cite[Proposition 5.3.7]{Bhatt}, except that one needs to ``twist by $(A,I)$". Let $S$ be a $p$-nilpotent test $\O_K$-algebra. By construction, the above composition takes a point $u\in \mathbf{A}_{-}^1(S)$ (viewed as a linear map $I\otimes_AS\xrightarrow{u}S$, as above) to the filtered Cartier--Witt divisor $M_u\xrightarrow{d_u}W$ determined by commutative diagram 
    \begin{displaymath}
         \begin{tikzcd}\label{diagram twisted version}
                    0\ar[r] & I\otimes_A\mathbf{G}_a^{\mathrm{\sharp}}\arrow[bend right=35,red]{dd}[near start,swap]{can}\ar[r]\ar[d,"u^{\mathrm{\sharp}}"] & I\otimes_A W\arrow[bend right=35,red]{dd}[near start,swap]{can}\ar[d]\ar[r] & I\otimes_A F_*W\arrow[bend left=35,red]{dd}[near start]{can} \ar[d,equal]\ar[r] & 0\\
    0\ar[r] & \mathbf{G}_a^{\mathrm{\sharp}}\ar[d,"0"]\ar[r] & M_u\ar[d,red,"d_{u}"] \ar[r] & I
    \otimes F_*W \ar[ld,"V\circ\beta",blue,hook]\ar[r]\ar[d,"can"] & 0\\
    0\ar[r] & \mathbf{G}_a^{\mathrm{\sharp}}\ar[r] & W\ar[r] & F_*W\ar[r] & 0.
        \end{tikzcd}
    \end{displaymath}
By the proof of \cite[Prop. 3.6.6]{BhattLurieabsolute} we have a factorization
\begin{displaymath}
    \begin{tikzcd}
            I\otimes_A W\ar[r,"F"]\arrow[bend right=20,red,"can"]{rrr} & I\otimes_A F_*W\ar[r,"\beta","\simeq"',blue] & F_*W\ar[r,hook,"V"] & W
        \end{tikzcd}
\end{displaymath}
for some isomorphism $\beta$\footnote{Note that $\beta$ is not induced by the trivialization $I=E(x)A\simeq A$!}. By diagram chasing it then follows that the map $M_u\xrightarrow{d_u}W$ factors as $M_u\twoheadrightarrow I\otimes F_*W\xhookrightarrow{V\circ\beta}W$. 

By definition and by the preceding paragraph, an $S$-point of $\mathbf{A}_{-}^1\times_{(\mathbf{Z}_p^{\N})_{t=0}\times\Spf(\O_K)}\mathbf{A}_{-}^1$ is a triple $(u,u',\iota)$ where $u,u'\in\mathbf{A}_{-}^1(S)$ and $\iota$ is a $W$-linear isomorphism $M_u\simeq M_{u'}$ commuting with the maps onto $I\otimes F_*W$. Consider the diagram
\begin{displaymath}
    \begin{tikzcd}
        0\ar[r] & \mathbf{G}_a^{\mathrm{\sharp}}\ar[d,"\simeq","\lambda^{-1}"']\ar[r] & M_{u}\ar[d,"\iota"',"\simeq"] \ar[r] & I
    \otimes F_*W \ar[r]\ar[d,equal] & 0\\
    0\ar[r] & \mathbf{G}_a^{\mathrm{\sharp}}\ar[r] & M_{u'}\ar[r] & I\otimes F_*W\ar[r] & 0.
    \end{tikzcd}
\end{displaymath}
The induced isomorphism on the left is then of the form $(\lambda^{-1})^{\mathrm{\sharp}}$ for a unique $\lambda\in S^\times$. One can then view $\iota$ as an isomorphism \textit{of extensions} 
\begin{displaymath}
     \begin{tikzcd}
        0\ar[r] & \mathbf{G}_a^{\mathrm{\sharp}}\ar[d,equal]\ar[r] & M_{\lambda^{-1}u}\ar[d,"\iota","\simeq"'] \ar[r] & I
    \otimes F_*W \ar[r]\ar[d,equal] & 0\\
    0\ar[r] & \mathbf{G}_a^{\mathrm{\sharp}}\ar[r] & M_{u'}\ar[r] & I\otimes F_*W\ar[r] & 0.
    \end{tikzcd}
\end{displaymath}
By Lemma \ref{elaboration lemma}, we must have $u'=a+\lambda^{-1}u$ for some (unique) $a\in\mathbf{G}_a^{\mathrm{\sharp}}(S)$ and $\iota=\iota(\alpha)$, the isomorphism \eqref{iso extensions elaboration}.

Thus there is an identification
\begin{align}\label{identification automorphism group}
 \mathbf{A}_{-}^1\times_{(\mathbf{Z}_p^{\N})_{t=0}\times\Spf(\O_K)}\mathbf{A}_{-}^1  \simeq (\mathbf{G}_a^{\mathrm{\sharp}}\rtimes\mathbf{G}_m)\times \mathbf{A}_{-}^1 
\end{align}
of groupoids over $\mathbf{A}_{-}^1$, where the action of $\mathbf{G}_a^{\mathrm{\sharp}}\rtimes \mathbf{G}_m$ on $\mathbf{A}_{-}^1$ is given by $(a,\lambda)\cdot_{\mathbf{Z}_p} u:=a+\lambda^{-1}u$. The composition $(\mathbf{A}_{-}^1)_{\O_K}\xrightarrow{\pi_{\O_K}} (\O_K^{\N})_{t=0}\to (\mathbf{Z}_p^{\N})_{t=0}\times \Spf(\O_K)$ therefore factors through a monomorphism
\begin{align*}
        (\mathbf{A}^1_{-}/\mathbf{G}_a^{\mathrm{\sharp}}\rtimes \mathbf{G}_m)_{\O_K}\hookrightarrow (\mathbf{Z}_p^{\N})_{t=0}\times\Spf(\O_K).
    \end{align*}
It remains to show that the map is surjective flat locally. Given a $p$-nilpotent $\O_K$-algebra $S$, by the proof of \cite[Prop. 5.3.7]{Bhatt} any $S$-point $(M\to W)\in (\mathbf{Z}_p^{\N})_{t=0}(S)$ arises as a composition $M\twoheadrightarrow I\otimes F_*W\xhookrightarrow{V\circ\beta}W$ where the first map is part of the extension $0\to \mathbf{V}(L)^{\mathrm{\sharp}}\to M\to I\otimes F_{*}W\to 0$ defining the admissible module $M$. Working locally one can trivialize the line bundle $L$, and then by Lemma \ref{elaboration lemma} this extension arises as the pushout of the standard sequence $0\to I\otimes_A \mathbf{G}_a^{\mathrm{\sharp}}\to I\otimes_A W\to I\otimes_A F_*W\to 0$ along some map $I\otimes_A\mathbf{G}_a^{\mathrm{\sharp}}\xrightarrow{u}\mathbf{G}_a^{\mathrm{\sharp}}$. This finishes the proof.      
\end{proof}
\begin{remark}\label{explicit iso remark}
We warn the reader the composition
\begin{align*}
    (\mathbf{A}_{-}^1)_{\O_K}\to (\mathbf{A}_{-}^1/\mathbf{G}_a^{\mathrm{\sharp}}\rtimes \mathbf{G}_m)_{\O_K}\xrightarrow[\simeq]{\pi_{\O_K}} (\mathbf{Z}_p^{\N})_{t=0}\times\Spf(\O_K)
\end{align*}
however does not agree with the analogous composition
\begin{displaymath}
(\mathbf{A}_{-}^1)_{\O_K}\to (\mathbf{A}_{-}^1/\mathbf{G}_a^{\mathrm{\sharp}}\rtimes \mathbf{G}_m)_{\O_K}\xrightarrow{\simeq} (\mathbf{Z}_p^{\N})_{t=0}\times\Spf(\O_K)    
\end{displaymath}
defined using the isomorphism from \cite[Prop. 5.3.7]{Bhatt}. In fact, this phenomenon can be already observed at the level of $\mathbf{Z}_p^{HT}$. Namely, recall from \cite[Prop. 5.1.4]{Bhatt} that the map $\eta:\Spf(\mathbf{Z}_p)\to\mathbf{Z}_p^{HT}$ given by the Cartier--Witt divisor $W(\mathbf{Z}_p)\xrightarrow{V(1)}W(\mathbf{Z}_p)$ induces an isomorphism $B\mathbf{G}_m^{\mathrm{\sharp}}\simeq \mathbf{Z}_p^{HT}$. On the other hand, given any Breuil--Kisin prism $(A,I)$ for $\mathbf{Z}_p$, the map $\overline{\rho}_{(A,I)}: \Spf(\mathbf{Z}_p)\to \mathbf{Z}_p^{HT}$ also induces an isomorphism $B\mathbf{G}_m^{\mathrm{\sharp}}\simeq\mathbf{Z}_p^{HT}$ (see e.g. the discussion in Subsection \ref{identify Sen} below). However, the maps $\eta$ and $\overline{\rho}_{(A,I)}$ do not coincide in general. For instance, for $(A,I):=(\mathbf{Z}_p[[x]],(x-p))$ one can check that they agree if and only if $p>2$.

We also note here that the description of $(\mathbf{Z}_p^{\N})_{t=0}$ from \cite[Prop. 5.3.7]{Bhatt} is completely canonical. This seems to be specific to the case $K=\mathbf{Q}_p$: for a general $K$ our similar description in Lemma \ref{Hodge-Tate description} requires the choice of a uniformizer. 
\end{remark}
\begin{remark} For future reference, we record here the isomorphism (of filtered Cartier--Witt divisors)
\begin{align*}
 \iota(a,\lambda): (M_u\xrightarrow{d_u} W)\simeq (M_{(a,\lambda)\cdot_{\mathbf{Z}_p}u} \xrightarrow{d_{(a,\lambda)\cdot_{\mathbf{Z}_p}u}} W),
\end{align*}
associated to an element $(a,\lambda)\in (\mathbf{G}_a^{\mathrm{\sharp}}\rtimes\mathbf{G}_m)(S)$ via the identification \eqref{identification automorphism group} above. It suffices to consider the factors $\mathbf{G}_a^{\mathrm{\sharp}}$ and $\mathbf{G}_m$ separately. First, the isomorphism $\iota(\lambda): M_u\simeq M_{\lambda^{-1}u}$ for $\lambda\in \mathbf{G}_m(S)$ is simply given by the pushout diagram
\begin{displaymath}
    \begin{tikzcd}
        0\ar[r] & \mathbf{G}_a^{\mathrm{\sharp}}\ar[d,"\lambda^{-1}"',"\simeq"]\ar[r] & M_u\ar[r]\ar[d,"\simeq",red] \ar[r] & I\otimes F_*W\ar[d,equal]\ar[r] & 0\\
        0\ar[r] & \mathbf{G}_a^{\mathrm{\sharp}}\ar[r] & M_{\lambda^{-1}u}\ar[r] & I\otimes F_*W\ar[r] & 0.
    \end{tikzcd}
\end{displaymath}
Now let $a\in \mathbf{G}_a^{\mathrm{\sharp}}(S)$. Recall that by its construction as a pushout, $M_u\simeq \tfrac{\mathbf{G}_a^{\mathrm{\sharp}}\oplus (I\otimes_A W)}{\{(u^{\mathrm{\sharp}}(x),-x)|x\in I\otimes \mathbf{G}_a^{\mathrm{\sharp}}\}}$ and similarly for $M_{a+u}$. The isomorphism $\iota(a): M_u\simeq M_{a+u}$ is then given by 
\begin{equation}\label{Ga sharp diagram}
    \begin{tikzcd}
        0\ar[r] & \mathbf{G}_a^{\mathrm{\sharp}}\ar[d,equal]\ar[r] & M_u\simeq \tfrac{\mathbf{G}_a^{\mathrm{\sharp}}\oplus (I\otimes_A W)}{\{(u^{\mathrm{\sharp}}(x),-x)|x\in I\otimes \mathbf{G}_a^{\mathrm{\sharp}}\}}\ar[d,"{(x,y)\mapsto (x-a(y),y)}","\simeq"',red]\ar[r] & I\otimes F_*W\ar[d,equal] \ar[r] & 0\\
        0 \ar[r] & \mathbf{G}_a^{\mathrm{\sharp}}\ar[r] & M_{a+u}\simeq\tfrac{\mathbf{G}_a^{\mathrm{\sharp}}\oplus (I\otimes_A W)}{\{(a+u)^{\mathrm{\sharp}}(x),-x)|x\in I\otimes\mathbf{G}_a^{\mathrm{\sharp}}\}}\ar[r] & I\otimes F_*W \ar[r] & 0.
    \end{tikzcd}
\end{equation}
\end{remark}
We are now ready to prove Proposition \ref{Hodge-Tate description}.
\begin{proof}[Proof of Proposition \ref{Hodge-Tate description}]
    Inspired by \cite[Construction 9.4]{BLprismatization}, we introduce the following auxiliary functor. Let $\F$ be the functor taking a $p$-nilpotent $\O_K$-algebra $S$ to the set of pairs $(u\in S,\tau)$, where $\tau$ is an $A$-linear map $I\to M_u$ making the diagram 
    \begin{equation}\label{tau}
        \begin{tikzcd}
            I\ar[d]\ar[r,"\tau"] & M_u\ar[d,"d_u"]\\
            A\ar[r] &  W
        \end{tikzcd}
    \end{equation}
    commute. Here as before $M_u\xrightarrow{d_u} W$ is the filtered Cartier--Witt divisor underlying the image of $u\in S=\mathbf{A}_{-}^1(S)$ under the map $\pi_{\O_K}$, which (by construction) is determined by the commutative diagram 
    \begin{equation}\label{big diagram pi_OK}
        \begin{tikzcd}
            0\ar[r] & I\otimes_A\mathbf{G}_a^{\mathrm{\sharp}}\ar[r]\ar[d,"u^{\mathrm{\sharp}}"] & I\otimes_A W\arrow[bend right=35,red]{dd}[near start,swap]{can}\ar[d,red,"\tau_{u,0}"]\ar[r] & I\otimes_A F_*W \ar[d,equal]\ar[r] & 0\\
    0\ar[r] & \mathbf{G}_a^{\mathrm{\sharp}}\ar[d,"0"]\ar[r] & M_u\ar[d,red,"d_u"] \ar[r] & I
    \otimes F_*W\ar[ld,"V\circ\beta",red,swap,hook] \ar[r]\ar[d,"can"] & 0\\
    0\ar[r] & \mathbf{G}_a^{\mathrm{\sharp}}\ar[r] & W\ar[r] & F_*W\ar[r] & 0.
        \end{tikzcd}
    \end{equation}
    Note that the middle column gives a map $I\xrightarrow{\tau_{u,0}}M_u$ which is an instance of the maps $\tau$ appearing in the definition of $\F$. This gives a map $(\mathbf{A}_{-}^1)_{\O_K}\to \F, u\mapsto (u,\tau_{u,0})$. Furthermore the commutative square \eqref{tau} can be viewed as a map of quasi-ideals, and thus induces a map $\bar{\tau}: A/I\to (W/M_u)(S)$ of animated rings. The assignment $(u\in S,\tau)\mapsto (M_u\xrightarrow{d_u} W,\bar{\tau})$ then defines a map $\F\to (\O_K^{\N})_{t=0}$ which clearly fits into 
    \begin{displaymath}
        \begin{tikzcd}
            (\mathbf{A}_{-}^1)_{\O_K}\arrow[bend left=30]{rr}{\pi_{\O_K}}\ar[r] & \F\ar[r] & (\O_K^{\N})_{t=0}. 
        \end{tikzcd}
    \end{displaymath}
    In particular the map $\F\to (\O_K^{\N})_{t=0}$ is also a surjection in the flat topology. 
    
    We next explain that the preferred element $\tau_{u,0}$ gives an identification $\F\simeq \mathbf{A}_{-}^1\times \mathbf{G}_a^{\mathrm{\sharp}}$. Indeed, as explained in the proof of Lemma \ref{twisted version lemma}, the map $d_u$ factors as $M_u\twoheadrightarrow I\otimes F_*W\xhookrightarrow{V\circ\beta} W$; in particular $\ker(d_u)$ identifies with $\mathbf{G}_a^{\mathrm{\sharp}}\hookrightarrow M_u$. Thus the assignment $\tau\mapsto \tau-\tau_{u,0}$ induces (for each fixed $u\in S$) a bijection between the set of $\tau$ making diagram \eqref{tau} commute and $\mathrm{Hom}_A(I,\mathbf{G}_a^{\mathrm{\sharp
    }})\simeq \mathbf{G}_a^{\mathrm{\sharp}}\{-1\}\simeq \mathbf{G}_a^{\mathrm{\mathrm{\sharp}}}$ (for the last identification we use again the chosen generator $E(x)$ of $I$). Under this identification, the map $(\mathbf{A}_{-}^1)_{\O_K}\to \F$ is simply $\mathbf{A}_{-}^1 \to \mathbf{A}_{-}^1\times\mathbf{G}_a^{\mathrm{\sharp}}, u\mapsto (u,0)$. 

    Now we compute $\F\times_{\O_K^{\N}}\F$. By definition, for a $p$-nilpotent $\O_K$-algebra $S$, $(\F\times_{\O_K^{\N}}\F)(S)$ is the groupoid of tuples $((u,\tau),(u',\tau');\iota,\kappa)$ where $(u,\tau),(u',\tau')\in \F(S)$; $\iota$ is an isomorphism 
\begin{displaymath}
    \begin{tikzcd}
        M_u\ar[dr,"d_u"']\ar[rr,"\iota","\simeq"'] && M_{u'}\ar[dl,"d_{u'}"]\\
        & W &
    \end{tikzcd}
\end{displaymath}
of filtered Cartier--Witt divisors, \textit{and} $\kappa$ is a homotopy
\begin{displaymath}
    \begin{tikzcd}
        A/I \arrow[r, bend left=50, ""{name=U, below},"\overline{\iota\circ\tau}"]
\arrow[r, bend right=50, ""{name=D,},"\overline{\tau}'"']
& (W/M_{u'})(S)
\arrow[rightarrow,from=U, to=D,"\kappa",swap]
    \end{tikzcd}
\end{displaymath}
from $\overline{\iota\circ \tau}$ to $\overline{\tau}'$ (recall that these are maps of animated rings). 

By Lemma \ref{twisted version lemma}, such isomorphisms $\iota$ are precisely of the forms $\iota=\iota(a,\lambda)$ for elements $(a,\lambda)\in (\mathbf{G}_a^{\mathrm{\sharp}}\rtimes\mathbf{G}_m)(S)$ such that $u'=(a,\lambda)\cdot_{\mathbf{Z}_p} u$. Moreover, using the explicit description of $\iota(a,\lambda)$ given in Remark \ref{explicit iso remark}, one checks that, under the identification $\F\simeq \mathbf{A}_{-}^1\times \mathbf{G}_a^{\mathrm{\sharp}}$ above, the action $((a,\lambda),(u,\tau))\mapsto ((a,\lambda)\cdot_{\mathbf{Z_p}}u,\iota(a,\lambda)\circ \tau)$ corresponds precisely to the action of $\mathbf{G}_a^{\mathrm{\sharp}}\rtimes \mathbf{G}_m$ on $\mathbf{A}_{-}^1\times\mathbf{G}_a^{\mathrm{\sharp}}$ given by the above action $\cdot_{\mathbf{Z}_p}$ on $\mathbf{A}_{-}^1$, and the following action 
\begin{align*}
 (a,\lambda)*x:=-a+\lambda^{-1}x   
\end{align*}
on $\mathbf{G}_a^{\mathrm{\sharp}}$. (Note the minus sign in $-a$(!); this comes precisely from the minus sign appearing in the formula of the isomorphism $\iota(a)$ in \eqref{Ga sharp diagram}.)  

It remains to consider the homotopy $\kappa$. The collection of such $\kappa$'s naturally identifies with the set of derivations $D\in\mathrm{Der}(A,\mathbf{G}_a^{\mathrm{\sharp}})$ such that $D|_I=\tau'-\iota\circ\tau$. This follows from the equivalence between quasi-ideals and DG algebras concentrated in degree $[-1,0]$ from \cite[\textsection 3.3]{drinfeldQuasi-ideal}. Note that under the identifications $\{\tau\}\simeq \mathbf{G}_a^{\mathrm{\sharp}}$ above and $\mathrm{Der}(A,\mathbf{G}_a^{\mathrm{\sharp}})\simeq \mathrm{Hom}_{\O_K}(\Omega^1_A\otimes_A \O_K,\mathbf{G}_a^{\mathrm{\sharp}})=\mathrm{Hom}_{\O_K}(\O_Kdx,\mathbf{G}_a^{\mathrm{\sharp}})\simeq \mathbf{G}_a^{\mathrm{\sharp}}$, the action $(D,\tau)\mapsto D|_I+\tau$ corresponds to the action of $\mathbf{G}_a^{\mathrm{\sharp}}$ on $\mathbf{G}_a^{\mathrm{\sharp}}$ given by $a\cdot_{\O_K}x:=E'(\pi)a+x$. 

In summary, we have shown that the map $\F\to (\O_K^{\N})_{t=0}$ factors through an isomorphism 
\begin{align*}
    \left(\tfrac{\mathbf{A}_{-}^1\times\mathbf{G}_a^{\mathrm{\sharp}}}{\mathbf{G}_a^{\mathrm{\sharp}}\rtimes (\mathbf{G}_a^{\mathrm{\sharp}}\rtimes \mathbf{G}_m)}\right)_{\O_K}\simeq (\O_K^{\N})_{t=0}.
\end{align*}
Here,
\begin{itemize}
    \item in the formation of the semidirect product $\mathbf{G}_a^{\mathrm{\sharp}}\rtimes (\mathbf{G}_a^{\mathrm{\sharp}}\rtimes \mathbf{G}_m)$, $\mathbf{G}_a^{\mathrm{\sharp}}\rtimes \mathbf{G}_m$ acts on $\mathbf{G}_a^{\mathrm{\sharp}}$ via the factor $\mathbf{G}_m$ and the multiplication action $(\lambda,a)\mapsto \lambda^{-1}a$ of $\mathbf{G}_m$ on $\mathbf{G}_a^{\mathrm{\sharp}}$;
    \item the factor $\mathbf{G}_a^{\mathrm{\sharp}}$ acts trivially on $\mathbf{A}_{-}^1$, and acts via $a\cdot_{\O_K}x:=E'(\pi)a+x$ on $\mathbf{G}_a^{\mathrm{\sharp}}$;
    \item the factor $\mathbf{G}_a^{\mathrm{\sharp}}\rtimes \mathbf{G}_m$ acts on $\mathbf{A}_{-}^1$ via $(a,\lambda)\cdot_{\mathbf{Z}_p}x:=a+\lambda^{-1}x$, and acts on $\mathbf{G}_a^{\mathrm{\sharp}}$ via $(a,\lambda)*x:=-a+\lambda^{-1}x$.  
\end{itemize}
Note that the above action $*$ of $\mathbf{G}_a^{\mathrm{\sharp}}\rtimes \mathbf{G}_m$ on $\mathbf{G}_a^{\mathrm{\sharp}}$ is transitive. In fact, under the embedding $\mathbf{G}_m^{\mathrm{\sharp}}\hookrightarrow \mathbf{G}_a^{\mathrm{\sharp}}\rtimes \mathbf{G}_m, \lambda\mapsto (1-\lambda^{-1},\lambda)$, the induced action of $\mathbf{G}_m^{\mathrm{\sharp}}$ on $\mathbf{G}_a^{\mathrm{\sharp}}$ is $(\lambda,a)\mapsto \lambda^{-1}(a+1)-1$, which is simply transitive (if we identify $\mathbf{G}_m^{\mathrm{\sharp}}=\mathbf{G}_a^{\mathrm{\sharp}}+1$ inside $W$, this is nothing but the multiplication action of the group scheme $\mathbf{G}_m^{\mathrm{\sharp}}$ on itself). 

Thus the map $\mathbf{A}_{-}^1\to  \mathbf{A}_{-}^1\times\mathbf{G}_a^{\mathrm{\sharp}}, u\mapsto (u,0)$ induces a surjection
\begin{align*}
    \mathbf{A}_{-}^1\twoheadrightarrow \tfrac{\mathbf{A}_{-}^1\times\mathbf{G}_a^{\mathrm{\sharp}}}{\mathbf{G}_a^{\mathrm{\sharp}}\rtimes (\mathbf{G}_a^{\mathrm{\sharp}}\rtimes \mathbf{G}_m)}.
\end{align*}
By unraveling the various actions, one then checks that this factors through an isomorphism
\begin{align*}
    (\mathbf{A}_{-}^1/\mathbf{G}_a^{\mathrm{\sharp}}\rtimes \mathbf{G}_m)_{\O_K}\simeq \left(\tfrac{\mathbf{A}_{-}^1\times\mathbf{G}_a^{\mathrm{\sharp}}}{\mathbf{G}_a^{\mathrm{\sharp}}\rtimes (\mathbf{G}_a^{\mathrm{\sharp}}\rtimes \mathbf{G}_m)}\right)_{\O_K}\simeq (\O_K^{\N})_{t=0},
\end{align*}
where in the LHS the action of $\mathbf{G}_a^{\mathrm{\sharp}}\rtimes \mathbf{G}_m$ on $\mathbf{A}_{-}^1$ is given by $(a,\lambda)\cdot_{\O_K}x:=E'(\pi)a+\lambda^{-1}x$. This finishes the proof. 
\end{proof}
\begin{remark}
    The same proof applies to give a similar description of $(R^{\N})_{t=0}$ for a general complete Noetherian regular local ring $R$ with perfect residue field of characteristic $p>0$. See \cite[Prop. 9.5]{BLprismatization} for the case of the open $(R^{\N})_{t=0}\cap j_{HT}(R^{\Prism})$. 
\end{remark}
\subsection{Identifying the Sen operator}
We now check that the lower diagram\footnote{Over $\Spf(\mathbf{Z}_p)$, $\mathbf{G}_m\subseteq\mathbf{A}_{-}^1$ is stable under the action $\cdot_{\O_K}$ of $\mathbf{G}_a^{\mathrm{\sharp}}\rtimes \mathbf{G}_m$, so it really makes sense to consider the quotient $\mathbf{G}_m/\mathbf{G}_a^{\mathrm{\sharp}}\rtimes\mathbf{G}_m$. Indeed, if $R$ is $p$-nilpotent and $a\in R$ admits divided powers, then $a^n\in n!R=0$ for $n\gg 0$, and so $a\cdot_{\O_K}x=E'(\pi)a+x\in R^\times$ for any $x\in R^\times$.}
\begin{displaymath}
    \begin{tikzcd}
     (\mathbf{A}_{-}^1/\mathbf{G}_m)_{\O_K}\ar[r,"can"] & (\mathbf{A}_{-}^1/\mathbf{G}_a^{\mathrm{\sharp}}\rtimes\mathbf{G}_m)_{\O_K}\ar[r,"\simeq"] & (\O_K^{\N})_{t=0}\\
    \Spf(\O_K)=(\mathbf{G}_m/\mathbf{G}_m)_{\O_K}\arrow[rr, bend right=13, "\overline{\rho}_{(\fkS,I)}"']\ar[u,hook,"u\ne 0"]\ar[r,"can"] & (\mathbf{G}_m/\mathbf{G}_a^{\mathrm{\sharp}}\rtimes\mathbf{G}_m)_{\O_K}\ar[u,hook,"u\ne 0"]\ar[r,"\simeq"] &\O_K^{HT}\ar[u,hook,"j_{HT}"]
    \end{tikzcd}
\end{displaymath}
from \eqref{diagram stack} is Cartesian. By tracing through definitions, it is easy to see that it commutes. For the Cartesian property, note that the middle column of diagram \eqref{big diagram pi_OK} defines a morphism $(I\otimes_A W\xrightarrow{can}W)\to (M_u\xrightarrow{d_u} W)$ of filtered Cartier--Witt divisors. Since $\O_K^{\Prism}\xhookrightarrow{j_{HT}} \O_K^{\N}$ identifies with the subgroupoid of invertible filtered Cartier--Witt divisors, we have by rigidity \cite[Lem. 5.1.5]{Bhatt} that, $(M_u\xrightarrow{d_u}W)\in j_{HT}(\O_K^{\Prism})$ if and only if $(I\otimes_A W\xrightarrow{can}W)\to (M_u\xrightarrow{d_u} W)$ is an isomorphism if and only if $I\otimes_A S\xrightarrow{u}S$ is an isomorphism, as wanted.\hfill $\square$

Our next goal is to identify the so-called Sen operator corresponding to the restriction $E|_{\O_K^{{HT}}}$; see Lemma \ref{identify Sen} below. This will explain our construction of $\Theta$ from $D$, as given in \eqref{intro relation D theta}. 

To this end, note that since the action $\cdot_{\O_K}$ of $\mathbf{G}_a^{\mathrm{\sharp}}\rtimes \mathbf{G}_m$ (or even just the subgroup $1\rtimes \mathbf{G}_m$) on $\mathbf{G}_m$ is transitive, the map $\overline{\rho}_{(\fkS,I)}:\Spf(\O_K)\to \O_K^{HT}$ factors through an isomorphism
\begin{displaymath}
    B(\mathrm{Stab}_{\mathbf{G}_a^{\mathrm{\sharp}}\rtimes \mathbf{G}_m}(1\in \mathbf{G}_a))\simeq  (\mathbf{G}_m/\mathbf{G}_a^{\mathrm{\sharp}}\rtimes\mathbf{G}_m)_{\O_K
    }\simeq \O_K^{HT}.
\end{displaymath}
Let $G_{\pi}:=\mathrm{Stab}_{\mathbf{G}_a^{\mathrm{\sharp}}\rtimes \mathbf{G}_m}(1\in \mathbf{G}_a)$. By definition
\begin{align*}
    G_{\pi}=\{(a,\lambda)\in \mathbf{G}_a^{\mathrm{\sharp}}\rtimes\mathbf{G}_m\;|\;E'(\pi)a+\lambda^{-1}=1\}.
\end{align*}
Via the projection $(a,\lambda)\mapsto a$, $G_{\pi}$ identifies with $\mathbf{G}_{a}^{\mathrm{\sharp}}$ as a formal scheme. The group structure on $G_{\pi}$ then transfers to the operation
\begin{align}\label{twisted action on Ga sharp}
    a\bullet b:=a+(1-E'(\pi)a)b
\end{align}
on $\mathbf{G}_a^{\mathrm{\sharp}}$. 

Thus pulling back along $\overline{\rho}_{(\fkS,I)}$ identifies quasi-coherent sheaves on $\O_K^{{HT}}$ with $p$-complete $\O_K$-modules $M$ equipped with a continuous coaction $M\to M\widehat{\otimes}_{\O_K}\O(\mathbf{G}_a^{\mathrm{\sharp}})$, where $\mathbf{G}_a^{\mathrm{\sharp}}\simeq G_{\pi}$ is equipped with the group structure given by \eqref{twisted action on Ga sharp} above. For such $M$, the Sen operator $\Theta_M: M\to M$ is defined as the infinitesimal action of the element $\epsilon\in \mathrm{Lie}(\mathbf{G}_a^{\mathrm{\sharp}})\subseteq \mathbf{G}_a^{\mathrm{\sharp}}(\O_K[\epsilon])$, i.e. $1+\epsilon\Theta_M$ is given by the composition
\begin{align*}
    M\to M\widehat{\otimes}_{\O_K} \O(\mathbf{G}_a^{\mathrm{\sharp}})\xrightarrow{a\mapsto \epsilon} M\otimes_{\O_K}\O_K[\epsilon]=M\oplus \epsilon M
\end{align*}
(where as before $a$ denotes the coordinate on $\mathbf{G}_a^{\mathrm{\sharp}}$). See \cite[\textsection 2.2]{arthurSenHodgeTate} for more details. Concretely, $\Theta_{M}$ is given by 
\begin{displaymath}
    \Theta_{M}: M\to M\widehat{\otimes}_{\O_K} \O(\mathbf{G}_a^{\mathrm{\sharp}})\xrightarrow{(d/da)|_{a=0}}M\otimes \O_K\simeq M.
\end{displaymath} 
\begin{thm}[{cf. \cite[Theorem 2.5]{arthurSenHodgeTate}}]\label{BL Sen}
    The functor
    \begin{align*}
        \mathrm{QCoh}(\O_K^{{HT}}) &\to \mathrm{Mod}_{\O_K[\Theta]}\\
        E &\mapsto (\eta^*E,\Theta_{\eta^*E})
    \end{align*}
    is fully faithful. Its essential image consists of those $M$ which are $p$-complete and for which the action of $\Theta^p-E'(\pi)^{p-1}\Theta$ on the cohomology $H^{\bullet}(k\otimes_{\O_K}^{{L}} M)$ is locally nilpotent\footnote{There is a similar equivalence for quasi-coherent complexes, but we will only need the result at the abelian level.}.
\end{thm}
\begin{lemma}\label{derivation}
    We have 
    \begin{displaymath}
     D_{\mathrm{qc}}((\mathbf{A}_{-}^1/\mathbf{G}_a^{\mathrm{\sharp}}\rtimes \mathbf{G}_m)_{\O_K})\simeq D_{\text{gr,$D$-nilp}}(\O_K\{u,D\}/(Du-uD-1)),  
    \end{displaymath}
    where $\deg(u)=-1$ (as above) and $\deg(D)=+1$.
\end{lemma}
Note that this implies a similar equivalence over $\Spf(\O_K)$ by restricting to $p$-complete objects on the RHS and requiring that $D$ is locally nilpotent mod $p$. 
\begin{proof}[Proof of the result at the abelian level]
We first show that given a $\O_K[u]=\O(\mathbf{A}_{-}^1)$-module $M$, the datum of an equivariant action of $\mathbf{G}_a^{\mathrm{\sharp}}$ on $M$ is equivalent to the datum of a locally nilpotent endomorphism $D: M\to M$ satisfying $Du-uD=E'(\pi)$. As in \cite[Proposition 2.4.4]{Bhatt}, giving a coaction $\mu: M\to M\otimes_{\O_K}\O(\mathbf{G}_a^{\mathrm{\sharp}})$ amounts to giving a locally nilpotent operator $D: M\to M$: given $D$, the corresponding coaction is $m\mapsto \sum_{i\geq 0}D^i(m)a^i/i!$. We check that the coaction is equivariant, or equivalently, it is linear over the ring map $\mu: \O_K[u]\to  \O_K[u]\otimes_{\O_K}\O(\mathbf{G}_a^{\mathrm{\sharp}}), u \mapsto u+E'(\pi)a$ if and only if $d$ satisfies $Du-uD=E'(\pi)$. For this, we compute
    \begin{displaymath}
        \mu(um)=\sum_{i\geq 0}D^i(um)a^i/i!,
    \end{displaymath}
    and 
    \begin{displaymath}
        \mu(u)\mu(m)=\sum_{i\geq 0}(u+E'(\pi)a)D^i(m)a^i/i!.
    \end{displaymath}
     By comparing the coefficients of $a^i/i!$, one deduces that $\mu(um)=\mu(u)\mu(m)$ if and only if $D^iu-uD^i=E'(\pi)iD^{i-1}$ for all $i\geq 1$ if and only if $Du-uD=E'(\pi)$, as claimed. 
It remains to incorporate a $\mathbf{G}_m$-action. Recall that giving a coaction $\mu_{\lambda}: M\to M\otimes_{\O_K} \O(\mathbf{G}_m)=M[\lambda^{\pm 1}]$ is the same as giving a grading $M=\bigoplus_n M^n$: $\mu_{\lambda}(m)=m\lambda^n$ for $m\in M^n$. Moreover, it is compatible with the action of $\mathbf{G}_m$ on $\mathbf{G}_a$ if and only if $u:M\to M$ is homogeneous of degree $-1$. Thus, we need to show that the two actions of $\mathbf{G}_a^{\mathrm{\sharp}}$ and $\mathbf{G}_m$ extend to an action of $\mathbf{G}_a^{\mathrm{\sharp}}\rtimes \mathbf{G}_m$ if and only if $D$ is homogeneous of degree $+1$. 

        The compatibility of the two actions is precisely the condition that $h(n(h^{-1}m))=(h\cdot n)(m)$ for all $h\in \mathbf{G}_m, n\in\mathbf{G}_a^{\mathrm{\sharp}}, m\in M$ and $h\cdot n$ denotes the action of $h$ on $n\in \mathbf{G}_a^{\mathrm{\sharp}}$. We can express the map $(h,n,m)\mapsto h(n(h^{-1}m))$ as 
        \begin{align*}
            \mathbf{G}_m\times \mathbf{G}_a^{\mathrm{\sharp}}\times M &\to \mathbf{G}_m\times \mathbf{G}_m\times \mathbf{G}_a^{\mathrm{\sharp}}\times M\to \mathbf{G}_m\times \mathbf{G}_a^{\mathrm{\sharp}}\times M\to \mathbf{G}_m\times M\to M\\
            (h,n,m) &\mapsto ((h,h^{-1}),n,m)\mapsto (h,n,h^{-1}(m))\mapsto (h,n(h^{-1}m))\mapsto h(n(h^{-1}m)).
        \end{align*}
        Translating in terms of coactions, this is given by the composition 
        \begin{align*}
            M &\to M\otimes \O(\mathbf{G}_m)\to M\otimes \O(\mathbf{G}_a^{\mathrm{\sharp}})\otimes\O(\mathbf{G}_m)\to M\otimes\O(\mathbf{G}_a^{\mathrm{\sharp}})\otimes \O(\mathbf{G}_m)\otimes \O(\mathbf{G}_m)\to M\otimes \O(\mathbf{G}_a^{\mathrm{\sharp}})\otimes \O(\mathbf{G}_m)\\
            m\in M^n &\mapsto m\otimes \lambda^n\mapsto \sum_{i}D^i(m)a^i/i! \otimes \lambda^n\mapsto \sum_i \mu_{\lambda'}(D^i(m))a^i/i!\otimes \lambda^n\mapsto \sum_i (\mu_{\lambda'}(D^i(m))|_{\lambda':=\lambda^{-1}})a^i/i!\otimes \lambda^n,
        \end{align*}
        where as above $\mu_t: M\to M[\lambda^{\pm 1}]$ records the action of $\mathbf{G}_m$ on $M$ and $\lambda':=\lambda^{-1}$. On the other hand, in terms of coactions, the map $(h,n,m)\mapsto (h\cdot n)(m)$ is given by the composition 
        \begin{align*}
            M\to M\otimes \O(\mathbf{G}_a^{\mathrm{\sharp}})\to M\otimes \O(\mathbf{G}_a^{\mathrm{\sharp}})\otimes \O(\mathbf{G}_m)\\
            m\mapsto \sum_i D^i(m)a^i/i!\mapsto \sum_i D^i(m)(\lambda^{-1}a)^i/i!. 
        \end{align*}
        Again by comparing coefficients of $a^i/i!$, we see that the two composition maps agree if and only if 
        \begin{displaymath}
            \lambda^n\mu_{\lambda'}(D^i(m))|_{\lambda':=\lambda}=\lambda^{-i}D^i(m)\Longleftrightarrow \mu_{\lambda}(D^i(m))=\lambda^{n+i}D^i(m)\in M[\lambda^{\pm 1}]
        \end{displaymath}
        for all $m\in M^n$. This happens if and only if $D^i(m)\in M^{n+i}$(recall that $M^n=\mu_{\lambda}^{-1}(M\lambda^n)$), i.e. $D$ is of homogeneous degree $+1$, as claimed. 
\end{proof}

\begin{lemma}[Identifying the Sen operator]\label{identify Sen}
Consider the open immersion
\begin{displaymath}
    j_{\mathrm{HT}}: \O_K^{{HT}}\simeq B\mathrm{Stab}_{\mathbf{G}_a^{\mathrm{\sharp}}\rtimes \mathbf{G}_m}(1)\simeq (\mathbf{G}_m/\mathbf{G}_a^{\mathrm{\sharp}}\rtimes \mathbf{G}_m)\hookrightarrow (\mathbf{A}_{-}^1/\mathbf{G}_a^{\mathrm{\sharp}}\rtimes \mathbf{G}_m)\simeq (\O_K^{\N})_{t=0}.
\end{displaymath}
Let $E$ be a quasi-coherent sheaf on $(\O_K^{\N})_{t=0}$. Let $M$ be the graded $\O_K\{u,D\}/(Du-uD-E'(\pi))$-module corresponding to $E$ under the identification in Lemma \ref{derivation}. By the Rees construction, $M$ corresponds to an increasing filtration $\mathrm{Fil}_{\bullet}$ (with transition maps $u: \mathrm{Fil}_i\to\mathrm{Fil}_{i+1}$) of $p$-complete $\mathbf{Z}_p$-modules together with a map $D:\mathrm{Fil}_{\bullet}\to \mathrm{Fil}_{\bullet}[-1]$ satisfying $Du-uD=E'(\pi)$. (Explicitly, $\mathrm{Fil}_{i}=M^{\deg=-i}$\footnote{Recall again our convention that $\deg(u)=-1$.}.) 

Then, under the identification in Theorem \ref{BL Sen}, the restriction $E|_{\O_K^{{HT}}}$ corresponds to the $\O_K$-module given by the underlying non-filtered module $\varinjlim_{i} \mathrm{Fil}_i$ together with the Sen operator given by $\Theta=(uD-E'(\pi)i)$ on $\mathrm{Fil}_i$.
\end{lemma}
\begin{proof}
    We will unwind the various identifications. First, the restriction of $E$ to $[\mathbf{G}_m/\mathbf{G}_a^{\mathrm{\sharp}}\rtimes \mathbf{G}_m]$ corresponds to the graded $\O_K[u,1/u]$-module $M[1/u]$ equipped with the obvious extension of $D$, i.e., $D(m/u^i):=D(m)/u^i-E'(\pi)im/u^{i+1}$\footnote{One checks that $D$ indeed acts locally nilpotently mod $p$; this is related to the fact we have seen above that $\mathbf{G}_a^{\mathrm{\sharp}}\rtimes \mathbf{G}_m$ preserves $\mathbf{G}_m\subseteq \mathbf{G}_a$ and follows again from the fact that we are working in a $p$-complete setting.}). Now the restriction of $E$ to $B(\mathrm{Stab}_{\mathbf{G}_a^{\mathrm{\sharp}}\rtimes \mathbf{G}_m}(1))$ corresponds to the quotient module $N:=M[1/u]/(u-1)$ together with the induced action of the subgroup $\mathrm{Stab}_{\mathbf{G}_a^{\mathrm{\sharp}}\rtimes \mathbf{G}_m}(1)$.

We need to compute the Sen operator on $N$ in terms of the usual identification (from the Rees dictionary)
\begin{align}\label{identification N}
        \varinjlim_i \mathrm{Fil}_i & \simeq (M[1/u])^{\mathrm{deg}=0}\simeq M[1/u]/(u-1)=:N\\
        m\in \mathrm{Fil}_i & \mapsto m/u^i\nonumber.
\end{align}
Recall that $\mathbf{G}_a^{\mathrm{\sharp}}\simeq \mathrm{Stab}_{\mathbf{G}_a^{\mathrm{\sharp}}\rtimes \mathbf{G}_m}(1)$ via $a\mapsto (a,(1-E'(\pi)a)^{-1})$. The induced action of $\mathbf{G}_a^{\mathrm{\sharp}}\simeq \mathrm{Stab}_{\mathbf{G}_a^{\mathrm{\sharp}}\rtimes \mathbf{G}_m}(1)\subseteq \mathbf{G}_a^{\mathrm{\sharp}}\rtimes \mathbf{G}_m$ on $M[1/u]$ is thus given by the composition
\begin{align*}
   M[1/u] &\to M[1/u]\otimes \O(\mathbf{G}_a^{\mathrm{\sharp}}\rtimes \mathbf{G}_m)\xrightarrow{a\mapsto a, \lambda\mapsto (1-E'(\pi)a)^{-1}} M[1/u]\otimes \O(\mathbf{G}_a^{\mathrm{\sharp}})\\
   m  & \mapsto \sum_{i\geq 0}D^i(m)\tfrac{a^i}{i!} (1-E'(\pi)a)^{-(\deg(m)+i)}
\end{align*}
(where recall that $a$ and $\lambda$ respectively denote the coordinates on $\mathbf{G}_a^{\mathrm{\sharp}}$ and $\mathbf{G}_m$). After applying $(d/da)|_{a=0}$ we see that the Sen operator on $M[1/u]$ is given by $\Theta_{M[1/u]}(m)=D(m)+E'(\pi)(\deg(m))m$. In particular for $m/u^i\in (M[1/u])^{\deg=0}$ (so that $m\in \mathrm{Fil}_i$), we have 
\begin{align}\label{unravel theta m}
    \Theta_{M[1/u]}(m/u^i) =D(m/u^i)=(uD-E'(\pi)i)(m)/u^{i+1}\equiv (uD-E'(\pi)i)(m)/u^i\bmod{(u-1)M[1/u]}.
\end{align}
As the formation of the Sen operator is functorial, we have a commutative square 
\begin{displaymath}
\begin{tikzcd}
    M[1/u]\ar[d,twoheadrightarrow] \ar[r,"\Theta_{M[1/u]}"] & M[1/u]\ar[d,twoheadrightarrow,"\bmod{(u-1)}"]\\
    N \ar[r,"\Theta_N"] & N.
\end{tikzcd}
    \end{displaymath}
As $(uD-E'(\pi)i)(m)\in \mathrm{Fil}_i$, it follows from \eqref{unravel theta m} that, via the identification \eqref{identification N}, the Sen operator on $N$ is given by $\Theta_N=uD-E'(\pi)i$ on $\mathrm{Fil}_i$, as desired. (Equivalently, $\Theta_N=uD$ under the identification $N\simeq (M[1/u])^{\mathrm{deg}=0}$; this is the description used in \cite[\textsection 6.5.4, second bullet point]{Bhatt}.) 
\end{proof}
\subsection{Identifying the Nygaard filtration}
Consider again the flat cover
\begin{align*}
    \pi_{\O_K}: \R(E(x)^{\bullet}\fkS)\simeq \Spf(W(k)[[x]][u,t]/(ut-E(x)))/\mathbf{G}_m & \xrightarrow{}  \O_K^{\N}.
\end{align*}
\begin{lemma}\label{nygaard filtration stack}
    The filtration over $E(x)^{\bullet}\fkS$ associated (via the Rees dictionary) to the pullback $\pi_{\O_K}^*E$ is precisely the Nygaard filtration $\mathrm{Fil}^{\bullet}\varphi^*\fM:=\varphi^*\fM\cap E(u)^{\mathbf{Z}}\fM$ on $\varphi^*\fM$. 
\end{lemma}
\begin{proof}
We first check that the non-filtered module underlying the filtration $\pi_{\O_K}^*E$ is indeed $\varphi^*\fM$. To see this, note that the restriction of $\pi_{\O_K}: \R(E(x)^{\bullet}\fkS)\to \O_K^{\N}$ to the open locus $j_{dR}(\O_K^{\Prism})=(\O_K^{\N})_{t\ne 0}$ identifies the composition
\begin{displaymath}
\begin{tikzcd}
    \Spf(\fkS)\xrightarrow{F\circ\rho_{(\fkS,I)}} \O_K^{\Prism}\ar[r,hook,"j_{\mathrm{dR}}"] &\O_K^{\N}.
    \end{tikzcd}
\end{displaymath}
(Note the Frobenius twist!) This follows easily by unraveling the various constructions. This implies the claim since restricting to the open $\{t\ne 0\}$ amounts to passing to the underlying non-filtered module.

Consider now the commutative square (arising from the map of prisms $(\fkS,E(x))\to (A_{\inf},\xi), u\mapsto [\pi^{\flat}]$)
\begin{displaymath}
    \begin{tikzcd}
        \R(\xi^{\bullet}A_{\inf})\ar[d]\ar[r,"\pi_{\O_C}"] & \O_C^{\N}\ar[d]\\
        \R(E(x)^{\bullet}\fkS)\ar[r,"\pi_{\O_K}"] & \O_K^{\N}. 
    \end{tikzcd}
\end{displaymath}
By Lemma \ref{useful lemma identify filtration} above, the pullback $\pi_{\O_C}^*(E|_{\O_C^{\N}})$ corresponds to the (honest) filtration $\mathrm{Fil}^{\bullet}(\varphi^*M_{\inf}):=\varphi^*M_{\inf}\cap E(u)^{\mathbf{Z}}M_{\inf}$ on $\varphi^*M_{\inf}$, where $M_{\inf}:=\fM\otimes_{\fkS}A_{\inf}$. As the map $\fkS\to A_{\inf}$ is (classically) faithfully flat, we have a natural identification $\mathrm{Fil}^{\bullet}(\varphi^*M_{\inf})\simeq \mathrm{Fil}^{\bullet}(\varphi^*\fM)\otimes_{\fkS}A_{\inf}$. In summary, we have shown that the filtration given by $\pi^*E$ is an honest filtration on $\varphi^*\fM$, which, after base change along the faithfully flat map $\fkS\to A_{\inf}$, agrees with the filtration $\mathrm{Fil}^{\bullet}(\varphi^*\fM)\otimes_{\fkS}A_{\inf}$ on $\varphi^*\fM\otimes_{\fkS} A_{\inf}$. Hence it must agree with the filtration $\mathrm{Fil}^{\bullet}\varphi^*\fM$, as wanted. 
\end{proof}
\begin{cor}\label{conjugate filtration}
    The increasing filtration corresponding to the pullback of $E$ under the map 
    \begin{displaymath}
     (\mathbf{A}_{-}^1/\mathbf{G}_m)\to (\mathbf{A}_{-}^1/\mathbf{G}_a^{\mathrm{\sharp}}\rtimes \mathbf{G}_m)_{\O_K}\xrightarrow[\simeq]{\pi_{\O_K}} (\O_K^{\N})_{t=0}
    \end{displaymath}
    is precisely the conjugate filtration
    \begin{displaymath}
    \mathrm{Fil}_{\bullet}^{\mathrm{conj}}\fM_{HT}: \quad \ldots\hookrightarrow \underbrace{\fil{i-1}\varphi^*\fM/\fil{i}\varphi^*\fM}_{\mathrm{Fil}_{i-1}^{\mathrm{conj}}}\xhookrightarrow{\times E(x)} \underbrace{\fil{i}\varphi^*\fM/\fil{i+1}\varphi^*\fM}_{\mathrm{Fil}_i^{\mathrm{conj}}}\xhookrightarrow{} \ldots.
\end{displaymath}
\end{cor}
\begin{proof}
We have a commutative (even cartesian) diagram 
\begin{displaymath}
    \begin{tikzcd}
        \R(E(x)^{\bullet}\fkS)\simeq \Spf(W(k)[[x]][u,t]/(ut-E(x)))/\mathbf{G}_m\ar[r,"\pi_{\O_K}"] & \O_K^{\N}\\
        {(\mathbf{A}_{-}^1/\mathbf{G}_m)}_{\O_K}\ar[u,hook,"t=0"]\ar[r] & (\O_K^{\N})_{t=0}\ar[u,hook].
    \end{tikzcd}
\end{displaymath}
Since the conjugate filtration is by definition the associated graded of the $E(x)^{\bullet}\fkS$-filtration $\mathrm{Fil}^{\bullet}\varphi^*\fM$, the result follows from Lemma \ref{nygaard filtration stack} because restricting to the closed $\{t=0\}$ amounts to passing to the associated graded.  
\end{proof}
\subsection{Identifying the Hodge filtration}\label{subsection de Rham}
\begin{lemma}\label{de Rham commute}
Let $X$ be a bounded $p$-adic formal scheme. THen for any object $(A,I)\in X_{\Prism}$, the diagram 
    \begin{displaymath}
        \begin{tikzcd}
       \Spf(A/I)\times B\mathbf{G}_m\ar[d,hook,"u=0"]\ar[r,hook,"t=0"] & \Spf(A/I)\times \mathbf{A}_{+}^1/\mathbf{G}_m\ar[d,hook,"u=0"]\ar[r] & X\times \mathbf{A}_{+}^1/\mathbf{G}_m\ar[d,"i_{dR}"]\\
       \Spf(A/I)\times\mathbf{A}_{-}^1/\mathbf{G}_m\simeq \R(I^{\bullet}A)_{t=0} \ar[r,hook] &  \R(I^{\bullet}A)\ar[r,"\pi_X"] & X^{\N}
        \end{tikzcd}
    \end{displaymath}
    commutes. Here the left vertical map is induced by the map of filtered rings $I^{\bullet}A\to A/I$ where the target has the trivial filtration; and the right vertical map is the de Rham map from \cite[Construction 5.3.13]{Bhatt}.
\end{lemma}
\begin{proof}
Commutativity of the left square is clear. We now show commutativity of the right square after further composing with the map $X^{\N}\to\mathbf{Z}_p^{\N}$; the rest of the proof is left to the reader. Let $S$ be a $p$-nilpotent test $A/I$-algebra and let $t: L\to S$ be an $S$-point of $\Spf(A/I)\times \mathbf{A}^1/\mathbf{G}_m$. In terms of the usual moduli description of $\R(I^{\bullet}A)$, the image of $t$ under the left vertical map corresponds to the factorization $(I\otimes_AS\xrightarrow{u=0}L\xrightarrow{t}S)$ of the natural map. Then by construction of $\pi_X$ (see diagram \eqref{big diagram pi_X general}), the image of $t$ under $\Spf(A/I)\times \mathbf{A}^1/\mathbf{G}_m\to \R(I^{\bullet}A)\xrightarrow{\pi_X}X^{\N}\to\mathbf{Z}_p^{\N}$ is the filtered Cartier--Witt divisor 
\begin{displaymath}
 V(L)^{\mathrm{\sharp}}\oplus (I\otimes_A F_*W)\xrightarrow{(t^{\mathrm{\sharp}},V\circ\beta)} W,   
\end{displaymath}
where as before $\beta$ is the isomorphism fitting into
\begin{displaymath}
    \begin{tikzcd}
            I\otimes_A W\ar[r,"F"]\arrow[bend right=20,red,"can"]{rrr} & I\otimes_A F_*W\ar[r,"\beta","\simeq"',blue] & F_*W\ar[r,hook,"V"] & W.
        \end{tikzcd}
\end{displaymath}
Thus $V(L)^{\mathrm{\sharp}}\oplus (I\otimes_A F_*W)\xrightarrow{(t^{\mathrm{\sharp}},V\circ\beta)} W$ identifies with $V(L)^{\mathrm{\sharp}}\oplus F_*W\xrightarrow{(t^{\mathrm{\sharp}},V)}W$ as filtered Cartier--Witt divisors. We are done since by definition of the de Rham map, the latter is precisely the image of $t$ under the composition $\Spf(A/I)\times \mathbf{A}^1/\mathbf{G}_m\to X\times \mathbf{A}^1/\mathbf{G}_m\xrightarrow{i_{dR}}X^{\N}\to\mathbf{Z}_p^{\N}$. 
\end{proof}
We apply this for $X=\Spf(\O_K)$ and $(A,I)=(\fkS,(E(x))$, our fixed Breuil--Kisin prism: 
\begin{cor}\label{identfy Hodge filtration}
    The pullback of $E$ under the de Rham map 
    \begin{displaymath}
        (\mathbf{A}_{+}^1/\mathbf{G}_m)_{\O_K}\xrightarrow{i_{dR}}\O_K^{\N}
    \end{displaymath}
    corersponds (via the Rees dictionary) to the Hodge filtration $\fil{\bullet}_{H}\fM_{\mathrm{dR}}$. Moreover, there is a natural graded isomorphism 
    \begin{align*}
        \gr_{\bullet}^{\mathrm{conj}}\fM_{{HT}}\simeq \mathrm{gr}_H^{\bullet}\fM_{\mathrm{dR}}.
    \end{align*}
\end{cor}
\begin{proof}
    This follows from Lemma \ref{de Rham commute} and Lemma \ref{nygaard filtration stack} since the Hodge filtration is by definition the image of the Nygaard filtration on $\varphi^*\fM$ under the natural map $\varphi^*\fM\to \varphi^*\fM/E(x)\varphi^*\fM=\fM_{\mathrm{dR}}$.
\end{proof}
\section{Relation with the classical theory}\label{section GaoLiu}
We finish by briefly indicating a more explicit construction of the operators $D$ and $\Theta$, following the work \cite{GaoLiu} by Gao--Liu. We refer the reader to \textit{loc. cit.} for additional details.

By the work of Kisin, given a Breuil--Kisin module $\fM$ coming from a crystalline Galois lattice, there is a canonical monodromy operator $N: \fM\otimes_{\fkS}\O\to \fM\otimes_{\fkS}\O$ over the derivation $N_{\nabla}:=\lambda \tfrac{d}{dx}$\footnote{Kisin considers instead the derivation $\lambda x\tfrac{d}{dx}$ (with an additional factor $x$) to accommodate the case of semistable representations, but this will not concern us.} on $\O$. Here as usual $\O$ denotes the ring of functions on the rigid open unit disk over $K_0$ (in the coordinate $x$), and $\lambda\in \O$ denotes the element $\prod_{n\geq 0}\varphi^n(E(x)/E(0))$. The construction of $N$ however only uses $\fM[1/p]$, so one may ask if it can be actually defined over an integral variant of $\O$. By the Dwork's trick, $\varphi$-modules over $\O$ extends uniquely to $\fkS\langle E(x)/p\rangle [1/p]$ (the ring of functions on the closed disc $\{|x|\leq |\pi|\}$) so we can also equivalently consider $N$ as being defined on $\fM\otimes_{\fkS}\fkS\langle E(x)/p\rangle [1/p]$ and linear over the derivation $N_{\nabla}:=E(x)\tfrac{d}{dx}$ on the coefficient ring. Note that there is now an obvious integral candidate, namely $S_{\max}:=\fkS\langle E(u)/p\rangle$, and one can ask if $N$ in fact extends to $\fM\otimes_{\fkS} S_{\max}$. By exploiting integral properties of the $G_K$-action on $\fM\otimes_{\fkS}A_{\inf}$, it is shown in \cite{GaoLiu} that this is indeed the case. (In \cite{bartlettCycle}, Bartlett also proves this result using similar arguments.) 

We claim that after extending scalars along the evalation map $\mathrm{ev}_{\pi}: S_{\max}\twoheadrightarrow \O_K$, $N$ recovers our Sen operator $\Theta$ on $M:=\fM_{HT}:=\fM/E(x)\fM$. To see this, recall that the map $\Spf(\O_K)\xrightarrow{\overline{\rho}_{(\fkS,I)}}\O_K^{{HT}}$ induces an isomorphism $B\mathbf{G}_a^{\mathrm{\sharp}}\simeq \O_K^{HT}$, and our $\Theta$ is then defined as the composition 
\begin{align*}
    \Theta_{M}: M\to M\widehat{\otimes}_{\O_K} \O(\mathbf{G}_a^{\mathrm{\sharp}})\xrightarrow{(d/da)|_{a=0}}M\otimes \O_K\simeq M,
\end{align*}
where the first map is the associated coaction map. On the other hand, similar to \cite[Prop. 3.2.8]{BhattLurieabsolute}, one can show that the natural diagram of stacks 
\begin{displaymath}
    \begin{tikzcd}
        \Spf(\fkS^{(1)}/I)\ar[d,"i_1"']\ar[r,"i_2"] & \Spf(\O_K)\ar[d,"\overline{\rho}_{(\fkS,I)}"]\\
        \Spf(\O_K)\ar[r,"\overline{\rho}_{(\fkS,I)}"] & \O_K^{HT}
    \end{tikzcd}
\end{displaymath}
is cartesian, where $\fkS^{(1)}$ denotes the self-coproduct of $(\fkS,I)$ as an object in the absolute prismatic site $(\O_K)_{\Prism}$. Thus, there is an identification $\fkS^{(1)}/I\simeq \O(\mathbf{G}_a^{\mathrm{\sharp}})= \widehat{\bigoplus}_{n\geq 0}\O_K\tfrac{a^n}{n!}$, and our $\Theta$ is also given by the composition 
\begin{align*}
   M\xrightarrow{can}M\otimes_{\O_K,i_1}(\O_K^{(1)}/I)\simeq M\otimes_{\O_K,i_2}(\O_K^{(1)}/I)=\widehat{\bigoplus}_{n\geq 0}M\tfrac{a^n}{n!}\xrightarrow{proj}Ma\simeq M,
\end{align*}
with the middle isomorphism being induced by the descent datum $\fM\otimes_{\fkS,i_1}\fkS^{(1)}\simeq \fM\otimes_{\fkS,i_2}\fkS^{(1)}$. This is what is called the ``prismatic Sen operator'' in \cite{GaoLiu}, and it is shown in Proposition 9.8 of \textit{loc. cit.} that this indeed agrees with the base change of Kisin's operator $N: \fM\otimes_{\fkS}S_{\max}\to \fM\otimes_{\fkS}S_{\max}$ along the map $ev_{\pi}: S_{\max}\twoheadrightarrow \O_K$, as claimed. 
\begin{remark}
    An advantage of the stacky approach is that the construction of $D$ and $\Theta$ works uniformly in the uniformizer $\pi$. In contrast, the argument in \cite{GaoLiu} (which makes crucial use of the Galois action on $\fM\otimes_{\fkS}A_{\inf}$) requires some additional care in the case $p=2$ (due to the usual issue that $\pi^{1/p}$ may belong to $K(\zeta_{p^{\infty}})$ in this case). 
\end{remark}
Above we have used Kisin's operator $N_{\nabla}$, but one can also use the monodromy operator in Breuil's theory to construct $\Theta$ (or equivalently $D$). We refer the reader to \cite[Lem. 2.3]{liuNygaard} for more details.

\medskip

\textsc{\small CNRS, IMJ-PRG, Sorbonne Universit\' e, 4 place Jussieu, 75005 Paris, France}\\
\indent\textit{Email address}: \href{mailto:phamd@imj-prg.fr}{\texttt{dat.pham@imj-prg.fr}}

\addcontentsline{toc}{section}{References}
\begin{thebibliography}{99}

\bibitem{arthurSenHodgeTate}
Johannes Ansch\" utz, Arthur-C\' esar Le Bras, and Ben Heuer, \textit{$v$-vector bundles on $p$-adic fields and Sen theory via the Hodge-Tate stack}, 2022. 

\bibitem{bartlettCycle}
Robin Bartlett, \textit{Cycles relations in the affine grassmannian and applications to Breuil--Mézard for G-crystalline representations}, 2023. 


\bibitem{BhattLurieabsolute}
Bhargav Bhatt and Jacob Lurie, \textit{Absolute prismatic cohomology}, 2022. 


\bibitem{Bhatt}
Bhargav Bhatt, \textit{Prismatic $F$-gauges}, 2022. Available at \href{https://www.math.ias.edu/~bhatt/teaching/mat549f22/lectures.pdf}{https://www.math.ias.edu/~bhatt/teaching/}
\bibitem{drinfeldprismatization}

Vladimir Drinfeld, \emph{Prismatization}, 2020. 

\bibitem{drinfeldQuasi-ideal}

Vladimir Drinfeld, \emph{On a notion of ring groupoid}, 2021. 

\bibitem{GaoLiu}
Hui Gao and Tong Liu, \textit{Integral Sen theory and integral Hodge filtration}, 2024. 

\bibitem{GLSunitary}
Toby Gee, Tong Liu, and David Savitt, \textit{The Buzzard-{D}iamond-{J}arvis conjecture for unitary groups}, 2014.
\bibitem{liuNygaard}
Tong Liu, \textit{Torsion graded pieces of Nygaard filtration for crystalline representation}, 2024.


\bibitem{BLprismatization}
Bhargav Bhatt and Jacob Lurie, \textit{The prismatization of $p$-adic formal schemes}, 2022.

\bibitem{prismatic}
Bhargav Bhatt and Peter Scholze, \textit{Prismatic $F$-crystals and crystalline Galois representations}, 2023.
\end{thebibliography}
\end{document}